\documentclass[11pt,reqno]{amsart}

\usepackage[text={155mm,225mm},centering]{geometry}  
\usepackage[mathscr]{eucal}
\usepackage[all]{xy}
\usepackage{mathrsfs,mathabx,xypic}
\usepackage{amsfonts,amsmath,amsthm}

\usepackage{amssymb}
\usepackage{bm}

\usepackage{latexsym,tabularx,graphicx,pict2e}
\usepackage{hyperref}

\DeclareMathAlphabet{\mathpzc}{OT1}{pzc}{m}{it}

\newtheorem{theorem}{Theorem}[section]
\newtheorem{lemma}[theorem]{Lemma}

\newtheorem{corollary}[theorem]{Corollary}
\newtheorem{proposition}[theorem]{Proposition}

\theoremstyle{definition}
\newtheorem{example}[theorem]{Example}
\newtheorem{definition}[theorem]{Definition}
\newtheorem{definition-lemma}[theorem]{Definition-Lemma}
\newtheorem{definition-theorem}[theorem]{Definition-Theorem}

\newtheorem{remark}[theorem]{Remark}

\newtheorem*{ack}{Acknowledgements}
\newtheorem*{convention}{Convention}

\allowdisplaybreaks

\title[Twisted bi-symplectic structure]{Twisted bi-symplectic structure on
Koszul twisted Calabi-Yau algebras}

\author[Chen]{Xiaojun Chen}
\address{Department of Mathematics, Sichuan University, Chengdu, 
Sichuan Province, 610064 P.R. China}
\email{xjchen@scu.edu.cn}

\author[Eshmatov]{Alimjon Eshmatov}
\address{Department of Mathematics and Statistics,
University of Toledo, Toledo, OH 43606, USA}
\email{alimjon.eshmatov@utoledo.edu}

\author[Eshmatov]{Farkhod Eshmatov}
\address{AKFA University, Tashkent 100095, Uzbekistan
and V.I. Romanovskiy Institute of Mathematics, Uzbekistan Academy of Sciences, 
Tashkent 100174, Uzbekistan}
\email{olimjon55@hotmail.com}

\author[Liu]{Leilei Liu}
\address{School of Mathematics (Zhuhai), Sun Yat-sen University, Zhuhai, 
Guangdong Province, 519082 P.R. China}
\email{liuleilei@mail.sysu.edu.cn}

\date{}

\begin{document}

\begin{abstract}
For a Koszul Artin-Schelter regular algebra (also called twisted Calabi-Yau algebra), 
we show that it has a ``twisted" bi-symplectic 
structure, which may be viewed as a noncommutative
and twisted analog of the shifted symplectic
structure introduced by Pantev, To\"en, Vaqui\'e and Vezzosi.
This structure gives a quasi-isomorphism between the tangent complex
and the twisted cotangent complex of the algebra,
and may be viewed as a DG enhancement
of Van den Bergh's noncommutative Poincar\'e duality; it
also induces a twisted symplectic 
structure on its derived representation schemes.


\end{abstract}

\maketitle

\setcounter{tocdepth}{1}\tableofcontents


\section{Introduction}

Let $A$ be an Artin-Schelter (AS for short) regular algebra
of dimension $n$. Then, denote by $\sigma$ the Nakayama
automorphism of $A$ and $A^e=A\otimes A^{\mathrm{op}}$
its enveloping algebra. 
Reyes, Rogalski and Zhang \cite{RRZ}
(see also Van den Bergh
\cite{VdBExist} and
Yekutieli and Zhang \cite{YZ}) proved
that there is an isomorphism
\begin{equation}\label{eq:inversedualizing}
\mathrm{RHom}_{A^e}(A, A\otimes A)\cong A_\sigma[n]
\end{equation}
in the derived category $D(A^e)$ of (left) $A^e$-modules,
where $A_\sigma$ is the twisted $A^e$-module whose underlying space
is $A$, with the action of $A^e$ given by $u\circ x\circ v:=\sigma(u)  xv$, 
for all $u, v, x\in A$.
Therefore, due to Van den Bergh \cite{VdB98},
the noncommutative Poincar\'e duality of $A$
reads as
\begin{equation}\label{eq:NCPoincareduality}
\mathrm{HH}^\bullet(A)\cong\mathrm{HH}_{n-\bullet}(A, A_\sigma),
\end{equation}
where $\mathrm{HH}^\bullet(A)$ and $\mathrm{HH}_{n-\bullet}(A, A_\sigma)$
are the Hochschild cohomology and the twisted Hochschild homology of $A$, respectively.
This paper aims to understand \eqref{eq:NCPoincareduality}, especially the role of the 
Nakayama automorphism, from the derived noncommutative geometry point of view.

\subsection{}
Before going to the details, let us look at the case where $\sigma$ is the
identity map. In this case, the AS-regular algebra is nothing but a Calabi-Yau algebra
in the sense of Ginzburg \cite{Ginzburg}
(therefore, an AS-regular algebra is also 
called a {\it twisted} Calabi-Yau algebra in \cite{RRZ}).
Then the isomorphism \eqref{eq:inversedualizing}
may be interpreted as the isomorphism between the (derived) noncommutative 
tangent and cotangent complexes of $A$, which gives rise to a 
{\it shifted noncommutative symplectic structure}
on $A$. The noncommutative symplectic structure was first introduced
by Crawley-Boevey, Etingof and Ginzburg in \cite{CBEG} for associative algebras 
and was called the {\it bi-symplectic structure},
which can easily be
generalized to the category of differential graded (DG) algebras and hence their homotopy category 
(see \cite{CE,Pridham} for some more details); it is also a noncommutative analog 
of the {\it shifted
symplectic structure} introduced by Pantev, To\"en, Vaqui\'e and
Vezzosi in \cite{PTVV}.

Now for an arbitrary AS-regular algebra,
it is very tempting for us to look for a version, possibly twisted, of the bi-symplectic
structure analogous to that of Calabi-Yau algebras; here, by ``twisted", we mean a structure
that encodes the Nakayama automorphism.

For simplicity, let us assume $A$ is also Koszul.
Denote by $A^{\textup{!`}}$ the Koszul dual coalgebra of $A$; then
its cobar construction $R:=\Omega(A^{\textup{!`}})$ is a DG free
cofibrant resolution of $A$. It is usually more convenient to consider 
the noncommutative geometric
structure on $R$ instead of $A$ in derived algebraic geometry 
(c.f. \cite{BFR,BKR,Yeung}). Our first theorem is:

\begin{theorem}\label{thm:firstmain}
Let $A$ be a Koszul AS-regular algebra of dimension $n$, and let $R$ be as above.
Then there exists a closed 2-form
$\bm\omega$ of total degree $n$ in the $\bm{twisted}$ Karoubi-de Rham complex $\widehat{\mathrm{DR}}^2_\sigma R$
of $R$
such that the contraction (see Lemma \ref{lemma:NCcontraction} for the definition)
\begin{equation}\label{iso:derivedtangenttoderivedcotangent}
\mathbf{\Psi}\circ\bm\iota_{(-)}\bm\omega: 
\widehat{ \mathbf D\mathrm{er}}\,R\stackrel{\simeq}\longrightarrow 
\widehat{\mathbf\Omega}^1_{\sigma}R[n-2]
\end{equation}
is an isomorphism of DG $R$-bimodules,
where $\widehat{ \mathbf D\mathrm{er}}\,R$ and $\widehat{\mathbf{\Omega}}^1_{\sigma}R$
are the noncommutative tangent complex and the $\bm{twisted}$
noncommutative cotangent complex of $R$ respectively.
\end{theorem}

In \S\ref{sect:NCdiffforms}, we shall give more 
details about the notations in the above theorem.
In the theorem, $\bm\omega$ is the twisted version of the shifted
bi-symplectic structure of a Koszul Calabi-Yau algebra and hence is called
a {\it twisted bi-symplectic structure}.
Moreover, the induced isomorphism on the
commutator quotient spaces on both sides of 
\eqref{iso:derivedtangenttoderivedcotangent}
is precisely Van den Bergh's
noncommutative Poincar\'e duality \eqref{eq:NCPoincareduality}.
In this sense, we may view \eqref{iso:derivedtangenttoderivedcotangent}
as a derived enhancement of \eqref{eq:NCPoincareduality}.

\subsection{}
 A guiding principle in the study of noncommutative geometry is via representations 
 due to Kontsevich and Rosenberg \cite{KR}.
It roughly says that a noncommutative geometric structure
on a noncommutative space (in this work, we mean an associative algebra),
if it exists, should induce its classical counterpart
on the affine schemes of its representations (called the {\it representation schemes}).
This principle has achieved much success,
for example, in the study of noncommutative Poisson
geometry and noncommutative symplectic geometry \cite{CB,CBEG,VdBDP}.

Later this principle is generalized
to the derived/homotopy setting by Berest, Khachatryan and Ramadoss in \cite{BKR}.
Such a generalization is highly nontrivial,
as it solves many issues 
that appear in the calculations of representation schemes
and reveals many new interesting structures. 
Going back to the Koszul AS-regular algebra case, we have

\begin{theorem}\label{thm:secondmain}
Let $A$ be a Koszul AS-regular algebra.
Denote by $\mathpzc{DRep}_V(A)$ 
the derived moduli stack
of representations of $A$ in a vector space $V$ (see \S\ref{sect:VdB} below
for more details). 
Then the twisted
bi-symplectic structure of $A$ induces a twisted
symplectic structure 
on $\mathpzc{DRep}_V(A)$.
\end{theorem}

Here, the {\it twisted symplectic structure}, parallel
to the noncommutative case, 
gives a quasi-isomorphism, up to degree shifting, between the tangent
and twisted cotangent complexes of $\mathpzc{DRep}_V(A)$.
It is the twisted
version of the shifted symplectic structure introduced
in \cite{PTVV} and also generalizes
some previous results obtained in \cite{CE,Pridham,Yeung}
for Calabi-Yau algebras.

The rest of the paper is devoted to proving the above two theorems.
It is organized as follows: 
in \S\ref{sect:NCdiffforms}, we first recall the noncommutative tangent and cotangent complexes
as well as their twisted version for DG algebras, and then introduce
the notion of twisted bi-symplectic structure;
in \S\ref{sect:ASalgebra}, we first recall the definition of AS-regular algebras
and then prove Theorem \ref{thm:firstmain};
in \S\ref{sect:DRep}, we briefly recall Berest et al.'s construction
of derived representation schemes, and in particular, study the differential forms
on them; in the last section, \S\ref{sect:applications}, after introducing
the twisted differential forms and twisted symplectic structure 
for DG commutative algebras, 
we prove Theorem \ref{thm:secondmain}.

\begin{ack}
This work is partially supported by NSFC
(No. 11671281, 11890660 and 11890663).
We are grateful to Youming Chen, Song Yang and Xiangdong Yang
for helpful conversations, and to the anonymous referee for 
suggestions which improve the presentation of the paper.
\end{ack}

\begin{convention}
Throughout the paper, $k$ is a field of characteristic zero. All algebras are unital over $k$,
and all tensors and homomorphisms are over $k$ unless
otherwise specified. Boldface letters such as $\mathbf D\mathrm{er}(-)$
and $\mathbf{\Omega}^\bullet(-)$ mean the corresponding noncommutative structures
for DG algebras while the usual letters such as $\mathrm D\mathrm{er}(-)$
and $\Omega^\bullet(-)$
mean those for DG commutative algebras.
\end{convention}

\section{Noncommutative differential calculus}\label{sect:NCdiffforms}

In this section we study the twisted analogue of the noncommutative
bi-symplectic structure;
its untwisted version was first introduced by
Crawley-Boevey, Etingof and Ginzburg 
in \cite{CBEG}; see also \cite{CE,Pridham} for the DG/derived case.

\subsection{Noncommutative differential forms}
Suppose $(R,\mu, \partial)$ is a
DG $k$-algebra, where $\mu$ is the multiplication and $\partial$
is the differential of degree $-1$.
The set of noncommutative 1-forms of $R$,
denoted by $\mathbf{\Omega}^1 R$,
is 
$$
\mathbb{\mathbf{\Omega}}^1 R:=\mathrm{ker}\{\mu:R\otimes R\to R\}.
$$
It is a DG $R$-bimodule generated by $r\otimes 1-1\otimes r\in R\otimes R$
for all $r\in R$, with the differential induced from $R$
and still denoted by $\partial$. 
Equivalently, $\mathbf{\Omega}^1 R$ is the $R$-bimodule
generated by $dr$, for all $r\in R$, 
subject to the following relations:
$d(r_1r_2)=(dr_1)r_2+r_1(dr_2)$, for all $r_1, r_2\in R$.
The identification of these two $R$-bimodules is given by
$$r\otimes 1-1\otimes r\mapsto dr.$$

In what follows, 
we always raise the degree of 1-forms up by 
1; in other words, by the set of noncommutative 1-forms of $R$ we shall mean
$\mathbf{\Omega}^1 R[-1]$ (also denoted by $\Sigma\mathbf{\Omega}^1 R$).
Elements like $r\otimes 1-1\otimes r\in\mathbf{\Omega}^1 R$ is now identified with $\Sigma (r\otimes 1
-1\otimes r)\in \mathbf{\Omega}^1 R[-1]$. With this convention,
the de Rham differential
$$d: R\to \mathbf{\Omega}^1 R[-1],\, r\mapsto \Sigma dr$$
has degree 1, and satisfies $d\circ \partial+\partial\circ d=0$.

\begin{remark}\label{remark:Yeungsaugmentation}
It is suggested by Yeung \cite{Yeung} that in practice it is better to take the
noncommutative 1-forms to be the cone
of the inclusion
$
i:\mathbf{\Omega}^1  R\to R\otimes R
$. 
Formally, we may view $R\otimes R$ as the set of noncommutative 1-forms
generated by $d\mathbf 1$, where $\mathbf 1$ is the image of $1\in k$ under the unit map.
Taking the degree shifting also into account, we denote
this set of augmented noncommutative 1-forms
by
$\widehat{\mathbf{\Omega}}^1 R[-1]$; more precisely,
\begin{equation}\label{eq:Yeungsnc1forms}
\widehat{\mathbf{\Omega}}^1 R[-1]=\mbox{cone}\{\mathbf{\Omega}^1  R\to R\otimes R\}[-1].
\end{equation}
In what follows, we adopt this convention,
which will be convenient for us to study the equivariant differential forms on
the representation schemes of $R$ 
(see Theorem \ref{thm:tangentandcotangentonDRep}).
\end{remark}

The set of {\it noncommutative de Rham forms} of $R$, denoted by $\widehat{\mathbf{\Omega}}^\bullet R$, 
is the tensor algebra
$$\widehat{\mathbf{\Omega}}^\bullet R:=T_R(\widehat{\mathbf{\Omega}}^1 R[-1])$$
with $d$ and $\partial$ extended to it by derivation and by letting $d^2=0$.
The triple $(\widehat{\mathbf{\Omega}}^\bullet R,\partial, d)$ is a mixed DG algebra.

\subsection{Double derivations}

According to \cite{CBEG,VdBDP}, the noncommutative vector fields
on a DG associative algebra are given by the {\it double derivations}.
By definition,
the space of double derivations $ \mathbf D\mathrm{er}\, R$ of $R$
is the set of derivations
$\mathrm{Der}(R, R\otimes R)$,
where the $R$-bimodule structure on $R\otimes R$ is 
the {\it outer $R$-bimodule} structure, given by
$$
u\circ (x\otimes y)\circ v:=ux\otimes yv,
$$
 for any $x, y, u, v\in R$.
Since the map $R\to \mathbf{\Omega}^1 R, x\mapsto dx$
is the universal derivation, meaning that every derivation of $R$ factors through 
$\mathbf{\Omega}^1 R$,
we have that
$$
 \mathbf D\mathrm{er}\, R\cong\mathrm{Hom}_{R^e}(\mathbf{\Omega}^1  R, R\otimes R).
$$
Observe that $R\otimes R$ also has an {\it inner $R$-bimoule} structure, which is given
by
\begin{equation}\label{eq:innerbimodulestr}
u * (x\otimes y)* v:=(-1)^{|u||x|+|v||y|+|u||v|}xv\otimes uy.
\end{equation}
With the inner $R$-bimodule structure on $R\otimes R$,
$ \mathbf D\mathrm{er}\, R$
is a DG $R$-bimodule.

Analogously to the 1-forms case (see Remark \ref{remark:Yeungsaugmentation}),
in what follows, we take the space of noncommutative vector fields to be
$$\widehat{ \mathbf D\mathrm{er}}\, R[1]:=\mathrm{Hom}_{R^e}(\widehat{\mathbf{\Omega}}^1 R[-1],
R\otimes R)$$ and let
$
T_R(\widehat{ \mathbf D\mathrm{er}}\, R[1])
$
be the tensor algebra generated by $\widehat{ \mathbf D\mathrm{er}}\, R[1]$
over $R$, which is called the space of {\it (noncommutative) polyvector fields} of $R$.

In what follows, we also call
$\widehat{\mathbf\Omega}^1 R[-1]$ and
$\widehat{\mathbf D\mathrm{er}}\, R[1]$
the noncommutative tangent 
and cotangent complexes of $R$ respectively.

\subsection{Twisted differential forms}
In the literature, an associative algebra together with
an automorphism is called a {\it twisted algebra}.
Twisted differential forms
and twisted derivations for twisted algebras
have been studied by many people; see, for example,
\cite{Andre,KSA,LSQ} and references therein.
In what follows, we briefly recall these concepts
for DG algebras.

Suppose $R$ is a DG algebra with an automorphism
$\sigma\in\mathrm{Aut}\, R$. 
The set of {\it twisted} noncommutative differential 1-forms
of $R$ is
\begin{equation}\label{def:twistednoncommdiffforms}
\mathbf{\Omega}^1_\sigma R
:= \ker\{\mu_\sigma: R\otimes R\to R,\, a\otimes b\mapsto a \sigma(b)\}.
\end{equation}
Here $\mu_\sigma$ is called the {\it twisted product} of $R$.
There is a close relationship between $\mathbf{\Omega}^1 R$ and
$\mathbf{\Omega}_\sigma^1 R$. To see this, let us
first give $\mathbf{\Omega}_\sigma^1 R$ an $R$-bimodule structure as follows:
\begin{equation}\label{eq:newtwistedbimodulestr}
u\circ\big(\sum_i x_i\otimes y_i\big)\circ v:=\sum_i\sigma(u) x_i\otimes y_i v,\quad
\mbox{for all}\, u, v\in R.
\end{equation}
Then it is straightforward to see that
the following map
\begin{equation}\label{def:Psi}
\mathbf\Psi: \mathbf{\Omega}^1 R\to 
\mathbf{\Omega}_\sigma^1 R,\, \sum_i x_i\otimes y_i\mapsto\sum_i\sigma(x_i)\otimes y_i
\end{equation}
is a map of DG $R$-bimodules,
where $\mathbf{\Omega}_\sigma^1 R$
is equipped with the above {\it twisted} $R$-bimodule structure 
\eqref{eq:newtwistedbimodulestr}.

Parallel to Remark \ref{remark:Yeungsaugmentation}, in what follows we 
take the twisted noncommutative 1-forms of $R$ to be
$$\widehat{\mathbf{\Omega}}^1_\sigma R[-1]:=\mbox{cone}\{\mathbf{\Omega}^1_\sigma R\rightarrow R\otimes R\}[-1],$$
where $R\otimes R$ is equipped with the twisted outer $R$-bimodule structure the same as
\eqref{eq:newtwistedbimodulestr}.
There is a natural map
\begin{equation}\label{def:twisteddifferential}
d_\sigma: R\to \mathbf{\Omega}_\sigma^1 R[-1]\subset
\widehat{\mathbf{\Omega}}^1_\sigma R[-1],\, a\mapsto \Sigma\big(\sigma(a)\otimes 1-1\otimes a\big),
\end{equation}
called the {\it twisted de Rham differential} of $R$. It satisfies the following {\it twisted
Leibniz rule}
$$d_\sigma(ab)=(d_\sigma a)\cdot b+(-1)^{|a|}\sigma(a)\cdot (d_\sigma b),
\quad\mbox{for all}\, a, b\in R.$$
Moreover, we also have $d_\sigma\circ\partial+\partial\circ d_\sigma=0: R\to\widehat{\mathbf{\Omega}}_\sigma^1 R[-1]$.

\subsection{The Karoubi-de Rham complex}

Recall that 
the {\it Karoubi-de Rham complex} of $R$, denoted by $\widehat{\mathrm{DR}}^\bullet R$,  
is $$\widehat{\mathrm{DR}}^\bullet R:=
\widehat{\mathbf{\Omega}}^\bullet R_\natural
:=\widehat{\mathbf{\Omega}}^\bullet R/[\widehat{\mathbf{\Omega}}^\bullet R,
\widehat{\mathbf{\Omega}}^\bullet R],$$
where $[-,-]$ means the super-commutator subspace.
It is equipped with two differentials $\partial$ and $d$ 
induced from $\widehat{\mathbf{\Omega}}^\bullet R$ (see, for example, \cite{CBEG}),
and hence forms a mixed complex.


Now suppose $(R, \sigma)$ is a twisted algebra.
Let
$$\widehat{\mathrm{DR}}_\sigma^\bullet R:=
\widehat{\mathbf{\Omega}}^\bullet R/\{xy-(-1)^{|x||y|}y
\sigma(x),\;\mbox{where}\; x,y\in\widehat{\mathbf{\Omega}}^\bullet R\},$$
which we would call the {\it twisted} Karoubi-de Rham complex of $R$.
Here $\sigma: \widehat{\mathbf{\Omega}}^\bullet R\to\widehat{\mathbf{\Omega}}^\bullet R$
is the automorphism which extends $\sigma: R\to R$.
It is related to the so-called {\it twisted} Hochschild homology of $R$.
In the literature, the notion of twisted Hochschild cohomology
was first introduced by Kustermans, Murphy and Tuset
in \cite{KMT}.
Soon after that,  Hadfield and Kr\"ahmer studied
the twisted Hochschild homology and computed
them for some twisted algebras in \cite{HK}.
They were used by Brown and Zhang
in the study of some twisted algebras
such as the AS-regular algebras in \cite{BZ}.
By definition, the twisted Hochschild homology of a twisted algebra $(A,\sigma)$
is the Hochschild homology of $A$ with values in $A_\sigma$.
The following theorem gives
a relationship between the twisted differential forms
and the twisted Hochschild homology.

\begin{theorem}\label{thm:relationstoHochschild}
Suppose $R$ is a DG free algebra with $\sigma\in\mathrm{Aut}\, R$. Then
\begin{equation}\label{eq:twistedHHKoszul}
\widehat{\mathrm{DR}}_{\sigma}^1 R\simeq
\overline{\mathrm{CH}}_\bullet(R, R_\sigma),
\end{equation}
where $\overline{\mathrm{CH}}_\bullet(R, R_\sigma)$
is the 
reduced twisted
Hochschild chain complex of $R$.
In particular, if $R\stackrel{\simeq}\to A$ is a DG free resolution of $A$,
then
$$
\widehat{\mathrm{DR}}_{\sigma}^1 R\simeq\overline{\mathrm{CH}}_\bullet(A, A_\sigma).
$$
\end{theorem}

Observe that in the above theorem, $\sigma$ also induces an automorphism of $A$,
so $A_\sigma$ makes sense. Moreover, if $\sigma$ is the identity map,
then the above quasi-isomorphism becomes
$$
\widehat{\mathrm{DR}}^1_\sigma R
=\widehat{\mathrm{DR}}^1 R\simeq\overline{\mathrm{CH}}_\bullet(R)\simeq\overline{\mathrm{CH}}_\bullet(A),
$$
whose proof can be found in \cite[Propositions 3.5]{CE}.
The proof for the twisted case is similar, 
so here we only give
a sketch of it, 
just for the sake of completeness; see also Liu \cite{Liu} for some similar discussions.

\begin{proof}[Proof of Theorem \ref{thm:relationstoHochschild}]
Suppose $R=k\langle x_1,\cdots, x_n\rangle$ with differential $\partial$.
If we denote $V:=\mathrm{Span}_k\{x_1,\cdots, x_n\}$, 
then $V\oplus k$ has a counital, $A_\infty$-coalgebra structure,
which is the Koszul dual coalgebra of $R$ (see \cite[\S1]{BB} for more details), 
and is denoted by $\tilde V$.
Now equip $R\otimes \tilde V\otimes R$ with the following differential $b$:
suppose the $A_\infty$-coproduct 
$\Delta_m (v)=\sum v_1\otimes\cdots\otimes v_m$, for $v\in \tilde V$,
then for $r_1\otimes v\otimes r_2\in R\otimes\tilde V\otimes R$,
\begin{align*}
b(r_1\otimes v\otimes r_2)=
&\partial(r_1)\otimes v\otimes r_2
+(-1)^{|r_1|+|v|}r_1\otimes v\otimes\partial(r_2)\\
&+\sum_m\sum_i (-1)^{|r_1|+|r_1|+\cdots+|r_{i-1}|}r_1\bar v_1\cdots \bar v_{i-1}\otimes v_i\otimes
\sigma(\bar v_{i+1}\cdots\bar v_m) r_2,
\end{align*}
where $\bar v_i$ is the image of $v\in\tilde V$ in $V$.
It is direct to see
$$\widehat{\mathbf{\Omega}}^1_{\sigma}R[-1]
=\mathrm{cone}\{\mathbf{\Omega}^1_{\sigma} R\to R\otimes R\}[-1]
\cong R\otimes \tilde V\otimes R,$$
and therefore we get
\begin{equation}\label{eq:qistwistedHH}
\widehat{\mathrm{DR}}^1_\sigma R
\cong\overline{\mathrm{CH}}_\bullet(\tilde V, \tilde V_\sigma).
\end{equation}
By Koszul duality theory for $A_\infty$-algebras (\cite[Theorem 1.2]{BB}) we have
\begin{equation}\label{eq:qistwistedHH2}
\overline{\mathrm{CH}}_\bullet(\tilde V, \tilde V_\sigma)\simeq
\overline{\mathrm{CH}}_\bullet(R, R_\sigma).
\end{equation}
(In loc. cit. the authors did not consider the twisted Hochschild homology, but this is straightforward.)
Thus combining \eqref{eq:qistwistedHH} and \eqref{eq:qistwistedHH2} we get
$$\widehat{\mathrm{DR}}^1_\sigma R\simeq\overline{\mathrm{CH}}_\bullet(R, R_\sigma).$$
Now suppose $R\stackrel{\simeq}\to A$ is a DG free resolution.
Then it induces a chain map $\overline{\mathrm{CH}}_\bullet(R, R_\sigma)\to
\overline{\mathrm{CH}}_\bullet(A, A_\sigma)$.
Filtering both complexes with respect to the number of components, we get
$$\overline{\mathrm{CH}}_\bullet(R, R_\sigma)\simeq\overline{\mathrm{CH}}_\bullet(A, A_\sigma)$$
by comparing the associated spectral sequences.
Combining it with the above quasi-isomorphism \eqref{eq:twistedHHKoszul}, 
we get the desired result.
\end{proof}

In \cite[Proposition 3.7]{CE} it is also proved that
\begin{equation}\label{eq:commutatorquotientofdoublederivation}
(\widehat{\mathbf D\mathrm{er}}\, R[1])_\natural
:=\widehat{\mathbf D\mathrm{er}}\, R[1]/\big[R,\widehat{\mathbf D\mathrm{er}}\, R[1]\big]
\simeq\overline{\mathrm{CH}}^\bullet(R)\simeq
\overline{\mathrm{CH}}^\bullet(A).
\end{equation}

\subsection{Twisted bi-symplectic structure}

From now on, we assume $R$ is a DG free algebra with $\sigma\in\mathrm{Aut}\, R$.
We introduce the {\it twisted} noncommutative contraction
of double derivations with differential forms, which is a generalization of
the noncommutative contraction introduced in \cite{CBEG}.

First, for any 
$\Theta\in\widehat{ \mathbf D\mathrm{er}}\, R$
and $a\in R$, in what follows
we write $\Theta(a)\in R\otimes R$ in the form
$\sum\Theta(a)'\otimes\Theta(a)''$.

Now for
$\bm\omega\in\widehat{\mathbf{\Omega}}^\bullet R$,
we say it is {\it $\sigma$-invariant} if $\sigma(\bm\omega)=\bm\omega$.
Suppose $\bm\omega=d a_1 d a_2\cdots d a_n\in \widehat{\bm\Omega}^n R$ is $\sigma$-invariant,
then
the {\it twisted contraction}
of $\Theta$ with $\bm\omega$ is given by
\begin{eqnarray*}
\bm\iota_{\Theta}\bm\omega:=\sum_{i=1}^n\sum (-1)^{\sigma_n}\Theta(a_i)''da_{i+1}\cdots da_n\sigma(da_1)\cdots\sigma(da_{i-1})
\sigma(\Theta(a_1)'),
\end{eqnarray*}
where $(-1)^{\sigma_n}$ is the Koszul sign.
We thus get an $R$-bilinear map 
$$\bm\iota_{\Theta}(-): \widehat{\mathbf{\Omega}}^\bullet R
\to \widehat{\mathbf{\Omega}}^{\bullet-1} R,\; \bm\omega\mapsto\bm\iota_{\Theta}\bm\omega.
$$ 
It is straightforward to see that $\bm\iota_{\Theta}(-)$ 
vanishes on the commutator subspace
$$\{xy-(-1)^{|x||y|}y
\sigma(x),\;\mbox{where}\; x,y\in\widehat{\mathbf{\Omega}}^\bullet R\}$$
and hence descends to a well-defined map
$$
\bm\iota_{\Theta}(-): \widehat{\mathrm{DR}}^\bullet_\sigma R\to \widehat{\mathbf{\Omega}}^{\bullet-1} R.
$$

Now recall that 
$\widehat{\mathrm{DR}}_\sigma^\bullet R$ is a mixed
complex with two differentials $\partial$ and $d$;
we say that $\bm\omega\in \widehat{\mathrm{DR}}_\sigma^\bullet R$
{\it extends to be a closed form} in the negative cyclic complex associated to
$\widehat{\mathrm{DR}}_\sigma^\bullet R$
if there exists a sequence 
$\bm\omega_0=\bm\omega,\bm\omega_1,\bm\omega_2,\cdots$
such that
$$\partial\bm\omega_0=0,\quad d\bm\omega_i=\partial\bm\omega_{i+1}\;\;
\mbox{for all}\; i=0,1,2,\cdots.$$

\begin{lemma}\label{lemma:NCcontraction}
If $\bm\omega\in\widehat{\mathrm{DR}}^2_\sigma R$ is $\sigma$-invariant and
extends to be a closed 2-form, then
the following map
\begin{equation}\label{eq:contractionwitha2form}
\bm\Psi\circ\bm\iota_{(-)}\bm\omega:\widehat{ \mathbf D\mathrm{er}}\, R
\to\widehat{\mathbf{\Omega}}_\sigma^1 R[n-2],\,\Theta\mapsto
\bm\Psi(\bm\iota_{\Theta}\bm\omega),
\end{equation}
where $\bm\Psi$ is given by \eqref{def:Psi},
is a map of DG $R$-bimodules.
\end{lemma}

\begin{proof}
Denote by $\mathbf\Phi$ the map $\Theta\mapsto\bm\Psi(\bm\iota_\Theta\bm\omega)$. 
We first show $\mathbf\Phi$ is a chain map.
In fact, in $\bm\omega=\sum d\eta'\otimes d\eta''$, if $\partial \eta'=\sum \eta'_1\cdots\eta'_p$
and $\partial \eta''=\sum \eta''_1\cdots\eta''_q$,
then
\begin{eqnarray}\label{formula:boundaryofomega}
\partial\bm\omega&=&
-\sum\big(d\partial \eta'\otimes d\eta''+(-1)^{|d\eta'|}d\eta'\otimes d\partial\eta''\big)\nonumber\\
&=&
-\sum\big(\sum_i(-1)^{|\eta_1'\cdots \eta_{i-1}'|}\eta'_1\cdots d\eta'_i\cdots\eta'_p
\otimes d\eta''\nonumber\\
&&+(-1)^{|d\eta'|}\sum_j(-1)^{|\eta''_1\cdots\eta''_{j-1}|}
d\eta'\otimes \eta''_1\cdots d\eta''_j\cdots \eta''_q\big).\nonumber
\end{eqnarray}
Since $R$ is free, that $\partial\bm\omega$ descends to be zero
in the quotient space $\widehat{\mathrm{DR}}^2_\sigma R$
means all the summands in the above expression cancel out, which
further means, up to a sign difference, the terms in the first summand and the terms
in the second summand are equal to each other up to
a cyclic permutation followed by a twisting; more precisely, we have,
up to a sign,
$
\sum_i d\eta'_i\cdots\eta'_p d\eta''\sigma(\eta'_1\cdots \eta'_{i-1})
$
equals
$
\sum_j
d\eta' \eta''_1\cdots d\eta''_j\cdots \eta''_q$
(see also Lemma 
\ref{thm:existenceoftwistedbisymplectic} for an explicit
computation when $p=q=2$).
Therefore if we denote by $\partial$ and $\delta$ the internal differential
(the differential induced from that of $R$)
of
$\widehat{\mathbf{\Omega}}^1_\sigma R$
and $\widehat{\mathbf{D}\mathrm{er}}\, R$ respectively, then 
on one hand,
\begin{eqnarray}
\partial\circ\Phi(\Theta)
&=&\partial\left(
\sum\Theta(\eta')''\cdot\mathbf\Psi(d\eta'')\cdot\sigma(\Theta(\eta')')\right)\nonumber\\
&=&\sum (\delta\Theta)(\eta')''\cdot\mathbf\Psi(d\eta'')\cdot\sigma\big((\delta\Theta)(\eta')'\big)\nonumber\\[2mm]
&&+\mbox{those terms involving $\partial\bm\omega$}\nonumber\\[2mm]
&=&\sum (\delta\Theta)(\eta')''\cdot\mathbf\Psi(d\eta'')\cdot\sigma\big((\delta\Theta)(\eta')'\big)
\label{eq:deltaofPhi},
\end{eqnarray}
where in the above, the third equality holds since
those terms involving $\partial\bm\omega$ vanish due to the twisted cyclicity described above.
On the other hand,
\begin{eqnarray}
\mathbf\Phi(\delta(\Theta))
&=&\sum (\delta\Theta)(\eta')''\cdot \mathbf\Psi(d\eta'')\cdot\sigma\big((\delta\Theta)(\eta')'\big)\label{eq:Phiofdelta}.
\end{eqnarray}
Thus that $\mathbf\Phi$ is a chain map follows
from that \eqref{eq:deltaofPhi} equals \eqref{eq:Phiofdelta}.

Next, we show that $\mathbf\Phi$ is a map of $R$-bimodules.
That is,
we need to show that, for any $u, v\in R$,
$$\mathbf\Phi(u\ast\Theta\ast v)
=u\cdot\mathbf\Phi(\Theta)\cdot \sigma(v),\quad\mbox{for all}\, \Theta\in\mathbf{D}\mathrm{er}\, R.
$$
In fact, 
if we write $\Theta(\eta)=\sum \Theta'(\eta)\otimes \Theta''(\eta)$,
then by \eqref{eq:innerbimodulestr} we have
$$(u\ast\Theta\ast v)(\eta)=\sum  \Theta'(\eta)\cdot v\otimes u\cdot \Theta''(\eta).$$
From this we obtain
\begin{eqnarray}
\mathbf\Phi(u\ast\Theta\ast v)
&=&\bm\iota_{u\ast\Theta\ast v}(\sum d\eta'\otimes d\eta'')\nonumber\\
&=&\sum
(u\cdot\Theta''(\eta'))\cdot \mathbf\Psi(d\eta'')\cdot \sigma(\Theta'(\eta')\cdot v))\nonumber\\
&=&u\cdot\left(\sum
\Theta''(\eta')\cdot \mathbf\Psi(d\eta'')\cdot \sigma(\Theta'(\eta')\right)\cdot \sigma(v)\nonumber\\
&=&u\cdot\mathbf\Phi(\Theta)\cdot \sigma(v).\label{eq:Rbimodulecompatibility}
\end{eqnarray}
The lemma is now proved.
\end{proof}

\begin{definition}[Twisted bi-symplectic structure]\label{def:twistedbisymplecticstru}
Let $R$ be a DG free algebra with $\sigma\in\mathrm{Aut}\, R$.
A twisted symmetric 2-form $\bm\omega\in\widehat{\mathrm{DR}}^2_{\sigma} R$ of total degree $n$
is called a {\it $(2-n)$-shifted twisted
bi-symplectic structure}
if 
\begin{enumerate}
\item it extends to be a closed 2-form in the negative cyclic complex associated to
$\widehat{\mathrm{DR}}^2_{\sigma} R$;

\item it is $\sigma$-invariant and induces a quasi-isomorphism
\begin{equation}\label{eq:contractionwith2forms}
\bm\Psi\circ\bm\iota_{(-)}\bm\omega:\widehat{ \mathbf D\mathrm{er}}\, R
\rightarrow\widehat{\mathbf{\Omega}}^1_\sigma R[n-2]
\end{equation}
of DG $R^e$-modules.
\end{enumerate}
We say a DG associative algebra $A$ has a {\it derived}
twisted bi-symplectic structure if its DG free model has a twisted bi-symplectic structure.
\end{definition}

\begin{remark}[About the degree shifting]
The $n$-shifted symplectic structure for derived stacks 
was introduced by Pantev et al in \cite{PTVV}.
The above definition is the noncommutative and twisted
analogue of theirs. 
In what follows, we prefer to write
\eqref{eq:contractionwith2forms}
in the form
$$
\bm\Psi\circ\bm\iota_{(-)}\bm\omega:\widehat{\mathbf D\mathrm{er}}\, R[1]
\rightarrow\widehat{\mathbf{\Omega}}^1_\sigma R[-1][n].
$$
\end{remark}
 
\begin{corollary}\label{cor:twistedsympimpliestwistedPD}
Suppose $A$ has a derived twisted $(2-n)$-shifted symplectic structure.
Denote by $\sigma$ the automorphism of $A$.
Then
$$
\mathrm{HH}^\bullet(A)\cong\mathrm{HH}_{n-\bullet}(A, A_\sigma).
$$
\end{corollary}

\begin{proof}
By \eqref{eq:twistedHHKoszul} and \eqref{eq:commutatorquotientofdoublederivation}
 we have
$$
(\widehat{\mathbf D\mathrm{er}}\, R[1])_\natural\simeq\overline{\mathrm{CH}}^\bullet(A)\quad\mbox{and}\quad
(\widehat{\mathbf{\Omega}}^1_\sigma R[-1])_\natural\simeq\overline{\mathrm{CH}}_\bullet(A,A_\sigma)
$$
respectively.
Thus the corollary follows from Definition \ref{def:twistedbisymplecticstru} and 
Theorem \ref{thm:relationstoHochschild}. 
\end{proof}

\section{AS-regular algebras}\label{sect:ASalgebra}

The notion of Artin-Schelter regular algebras was introduced
by Artin and Schelter in \cite{AS}. They have been extensively studied
in the past three decades.
One may refer to \cite{Yekutieli}, in particular Chapter 18, 
for some more backgrounds which
are related to our study.
In this section, we first recall the definition of AS-regular algebras,
discuss their Koszul dual algebra and coalgebra,
and then prove Theorem \ref{thm:firstmain}.

\begin{definition}[Artin-Schelter \cite{AS}]
A connected graded algebra $A$ is called {\it Artin-Schelter regular} (or AS-regular for short) of dimension $n$ if
\begin{enumerate}
\item $A$ has finite global dimension $n$, and
\item $A$ is Gorenstein, that is, 
$\mathrm{Ext}^i_{A}(k,A)=0$ for $i\neq n$ and $\mathrm{Ext}^n_{A}(k,A)\cong k$.
\end{enumerate}
\end{definition}

Every AS-regular algebra has an automorphism, called the {\it Nakayama automorphism}. 
In this paper, we are mostly interested in those AS-regular algebras which are Koszul
(c.f. \cite[Chapters 1-3]{LV} for more details of Koszul algebras).
For a Koszul AS-regular algebra, its Nakayama automorphism can be
explicitly described; 
for the Nakayama automorphism of a not-necessarily-Koszul AS regular algebra, see,
for example, \cite{RRZ}.

Suppose $A$ is a Koszul AS-regular algebra, then the Koszul complex
$$
\xymatrix{
\cdots\ar[r]&
A\otimes A^{\textup{!`}}_i\ar[r]^-d&
A\otimes A^{\textup{!`}}_{i-1}\ar[r]^-d&
\cdots\ar[r]&
A\otimes k\ar[r]^-{\varepsilon}&
k
}$$
gives a resolution of $k$ as a left $A$-module,
where $A^{\textup{!`}}$ is the Koszul dual coalgebra of $A$.
From this we obtain 
$$\mathrm{RHom}_A(k, A)=\mathrm{Hom}_A(\oplus_i A\otimes A^{\textup{!`}}_i, A)=A\otimes A^!,
$$
where $A^!$ is the Koszul dual algebra of $A$,
and thus from $\mathrm{RHom}_A(k, A)\cong k[n]$ we have
$$
A\otimes A^!\cong A\otimes A^{\textup{!`}}[n].
$$
The compatibility of the differentials gives
isomorphism $A^!\cong A^{\textup{!`}}[n]$ as left $A^!$-modules.
Similarly we get $A^!\cong A^{\textup{!`}}[n]$ as right $A^!$-modules.
These two isomorphisms together are equivalently saying that $A^!$ is a Frobenius algebra
of degree $n$
(see e.g. \cite{Smith} for more studies of this result).
In what follows, we study the structure of a Frobenius algebra with some
more details.

\subsection{Frobenius algebras}

In the literature, a {\it Frobenius algebra structure} of degree $n$
on a finite dimensional graded algebra, say $A^!$, is a non-degenerate
bilinear pairing of degree $n$
\begin{equation}\label{eq:defofFrobeniuspairing}
\langle-,-\rangle: A^!\otimes A^!\to k
\end{equation}
such that $\langle a, b\cdot c\rangle=\langle a\cdot b, c\rangle$,
for all $a, b, c\in A^!$. This is equivalent
to saying that $A^!$ is isomorphic to its linear dual $A^{\textup{!`}}[n]$ as
left and right $A^!$-modules.
Since $A^!$ is finite dimensional, $A^{\textup{!`}}$
is a graded coalgebra, and the adjoint of \eqref{eq:defofFrobeniuspairing}
gives a map
\begin{equation}\label{eq:adjointofthepairing}
k\to A^{\textup{!`}}\otimes A^{\textup{!`}}, \quad 1\mapsto \sum \eta'\otimes\eta''.
\end{equation}
This map factors through $A^{\textup{!`}}$; that is, there is an $\eta$ 
in the degree $n$ component of $A^{\textup{!`}}$ 
(which is 1-dimensional) 
such that $\sum\eta'\otimes\eta''=\Delta(\eta)$,
which $\Delta(-)$ is the coproduct of $A^{\textup{!`}}$.
In this form,
the pairing on $A^!$ is also given by
$$
(u, v)\mapsto \eta'(v)\cdot\eta''(u),\quad\mbox{for}\,\, u, v\in A^!,
$$
By the non-degeneracy of the pairing on $A^!$,
we have two isomorphisms of graded vector spaces
$
A^!\to A^{\textup{!`}}[n]
$
which are given by
\begin{equation}\label{eq:modulemap}
u\mapsto \sum \eta''(u)\cdot \eta'\quad\mbox{and}\quad
u\mapsto \sum \eta'(u)\cdot\eta''.
\end{equation}
Since $\sum \eta'\otimes\eta''=\Delta(\eta)$,
these two maps are exactly the slant product of $u$ with $\eta$ from
the left and the right respectively,
and therefore by the co-associativity of the co-product,
these isomorphisms can be extended to 
an isomorphism of left and right $A^!$-modules, respectively.

However, on the other hand, these two isomorphisms
may not give an $A^!$-bimodules on $A^{\textup{!`}}$, due to the fact that
the pairing on $A^!$ may not be graded symmetric; that is, $\langle a, b\rangle$ may not
be equal to $(-1)^{|a||b|}\langle b,a\rangle$.
Nevertheless, by the non-degeneracy of the pairing
there exists an automorphism $\sigma^*: A^!\to A^!$ such that 
$\langle a, b\rangle=(-1)^{|a||b|}\langle b,\sigma^*(a)\rangle$.
Now, equip
$A^{\textup{!`}}$ as an $A^!$-bimodule
structure by setting
$$u\circ x\circ v:=\sigma^*(u) x v,$$
which is denoted by $A^{\textup{!`}}_\sigma$, 
then we have
$
A^!\cong A^{\textup{!`}}_\sigma[n]
$
as $A^!$-bimodules.
Here, $\sigma^*$ is called the {\it Nakayama automorphism} of $A^!$,
whose adjoint is denoted by $\sigma:A^{\textup{!`}}\to A^{\textup{!`}}$.

Now let  $\Omega(A^{\textup{!`}})$ be the cobar construction of $A^{\textup{!`}}$, which
is a DG free algebra generated by $\Sigma^{-1}\bar A^{\textup{!`}}$,
the augmentation of $A^{\textup{!`}}$ with degree shifted down by 1.
Then $\sigma$ can be lifted to be an automorphism
$\sigma: \Omega(A^{\textup{!`}})\to  \Omega(A^{\textup{!`}})$.
Going back to the Koszul AS-regular algebra case, 
since $\Omega(A^{\textup{!`}})\to A$ is a quasi-isomorphisms of DG algebras,
$\sigma$ also induces an automorphism $\sigma: A\to A$, which
is also called the {\it Nakayama automorphism} of $A$.
We have the following result due to Reyes, Rogalski and Zhang in \cite[Lemma 1.2]{RRZ}:

\begin{theorem}Let $A$ be an associative algebra.
Then $A$ is AS-regular of dimension $n$ with Nakayama automorphism
$\sigma$ if and only if
$A$ is twisted Calabi-Yau of dimension $n$; that is, $A$ satisfies the following two
conditions:
\begin{enumerate}
\item $A$ is homologically smooth, and
\item $\mathrm{RHom}_{A^e}(A, A\otimes A)\cong A_{\sigma}[n]$ in $D(A^e)$.
\end{enumerate}
\end{theorem}

In the theorem, {\it twisted Calabi-Yau} comes from the fact that, if $\sigma$ is the identity map, then
$A$ is exactly a {\it Calabi-Yau algebra} in the sense of Ginzburg \cite{Ginzburg}.
Observe that the second condition in the above theorem
is a derived invariant, and therefore if $R$ is a DG free resolution of $A$,
then we also have
$\mathrm{RHom}_{R^e}(R, R\otimes 
R)\cong R_{\sigma}[n]$ in $D(R^e)$,
where $\sigma$ is the lifting of the Nakayama automorphism of $A$.
Since now $R$ is DG free, we have a short exact sequence
$$
0\longrightarrow \mathbf{\Omega}^1 R\longrightarrow R\otimes R\longrightarrow R\longrightarrow 0,
$$
where $\mathbf{\Omega}^1 R$ is a DG free $R$-bimodule.
This means $\widehat{\mathbf{\Omega}}^1 R[-1]$ is a DG free resolution of $R$ as $R$-bimodules,
and therefore,
$$\mathrm{RHom}_{R^e}(R, R\otimes R)\cong\mathrm{Hom}_{R^e}(\widehat{\mathbf{\Omega}}^1 R[-1], R\otimes R)
\cong\widehat{\mathbf{D}\mathrm{er}}\, R[1]$$
and similarly
$
\widehat{\mathbf{\Omega}}^1_\sigma R[-1]\cong R_\sigma[n]
$ 
in $D(R^e)$. Thus
$\mathrm{RHom}_{R^e}(R, R\otimes 
R)\cong R_{\sigma}[n]$ 
is equivalent to 
$
\widehat{\mathbf{D}\mathrm{er}}\, R[1]\cong \widehat{\mathbf{\Omega}}^1_\sigma R[-1][n]
$
in $D(R^e)$.

We now discuss some more details on $\widehat{\mathbf{D}\mathrm{er}}\, R[1]$ for 
$R=\Omega(A^{\textup{!`}})$.
Since $A^!$ is finite dimensional (and hence so is $A^{\textup{!`}}$), the product 
$\mu: A^!\otimes A^!\to A^!$ induces
two maps
$$
A^!\to A^!\otimes A^{\textup{!`}}, u\mapsto \sum_{c} \tilde cu\otimes c\quad \mbox{and}\quad
A^!\to A^{\textup{!`}}\otimes A^!, u\mapsto \sum_{c} c\otimes u\tilde c 
$$
by adjunction, where $c$ runs over the bases of $A^{\textup{!`}}$ and $\tilde c$ is its linear dual.
By the associativity of the product, the above two maps
in fact make $A^!$ into a left and right $A^{\textup{!`}}$-comodules.
Moreover, since the product on $A^!$ is an $A^!$-bimodule map, we get that
$$A^!\mapsto A^!\otimes A^{\textup{!`}}\oplus A^{\textup{!`}}\otimes A^!\subset A^{\textup{!`}}\otimes A^!\otimes A^{\textup{!`}},
\, u\mapsto \sum_c \tilde cu\otimes c+c\otimes u\tilde c $$
makes $A^!$ into an $A^{\textup{!`}}$-bi-comodule.
The internal differential, say $\delta$, on $\widehat{\mathbf{D}\mathrm{er}}\, R$
is given as follows:
since 
\begin{equation}\label{eq:adjointofderandforms}
\widehat{\mathbf{D}\mathrm{er}}\, R[1]\cong \mathrm{Hom}_{R^e}(\widehat{\mathbf{\Omega}}^1 R[-1],R\otimes R)
\cong \mathrm{Hom}_{R^e}(R\otimes A^{\textup{!`}}\otimes R, R\otimes R)\cong R\otimes A^!\otimes R
\end{equation}
as $R^e$-modules,
taking a generator $u\in A^!\cong k\otimes A^!\otimes k\subset R\otimes A^!\otimes R$,
if we denote by $\bm Du\in\widehat{\mathbf{D}\mathrm{er}}\, R[1]$ its
image under the above isomorphism, 
then
\begin{equation}\label{eq:internaldiffonDer}
\delta (\bm Du)=\sum_c(-1)^{|c|+|u|-1} \tilde c u\otimes \bar c+\bar c\otimes u\tilde c\in
A^!\otimes \Sigma^{-1}\bar A^{\textup{!`}}\oplus \Sigma^{-1} \bar A^{\textup{!`}}\otimes A^!\subset R\otimes A^!\otimes R,
\end{equation}
where $\bar c$ is image of $c$ under the natural map $A^{\textup{!`}}\to\Sigma^{-1} \bar A^{\textup{!`}}$.
Observe that $\delta$ is dual, via \eqref{eq:adjointofderandforms}, 
to the internal differential 
on $\widehat{\mathbf{\Omega}}^1 R$ induced by $\partial$ on $R$.

\subsection{The twisted bi-symplectic structure}\label{subsect:twistedbisymp}

Now suppose $A$ is a Koszul AS-regular algebra of dimension $n$
and let $R:=\Omega(A^{\textup{!`}})$.
Take $\sum \eta'\otimes\eta''=\Delta(\eta)$ as given in \eqref{eq:adjointofthepairing}
and
consider the following 2-form
\begin{equation}\label{eq:bisymplecticform}
\bm\omega:=\frac{1}{2}\sum d \eta'\otimes d \eta''\in\widehat{\mathbf{\Omega}}^2 R,
\end{equation}
which descends to a 2-form, still denoted by $\bm\omega$, 
in $\widehat{\mathrm{DR}}^2_{\sigma} R$.
We have the following:

\begin{lemma}\label{thm:existenceoftwistedbisymplectic}
Suppose $A$ is Koszul AS-regular. Let $R=\mathbf{\Omega}(A^{\textup{!`}})$ be the cobar construction of $A^{\textup{!`}}$.
Then $\bm\omega\in\widehat{\mathrm{DR}}^2_{\sigma} R$ 
is both $\partial$- and $d$-closed.
\end{lemma}

\begin{proof} (Compare with \cite[Lemma 4.5]{CE}.)
It is obvious that $\bm\omega$ is $d$-closed. We next show it is also $\partial$-closed.
In fact, if we write $\Delta(c)=\sum c_1\otimes c_2$ for any $c\in A^{\textup{!`}}$, then
\begin{eqnarray}
\partial \bm\omega 
&=&\frac{1}{2}\sum \partial(d \eta') \otimes d \eta'' +(-1)^{|\eta'|}d \eta' \otimes
\partial(d \eta'')\nonumber \\
&=&-\frac{1}{2}\sum d(\partial \eta')\otimes d \eta''+(-1)^{|\eta'|}d \eta'\otimes d(\partial  \eta'')
\nonumber\\
&=&-\frac{1}{2}\sum (-1)^{|\eta'_1|}d( \eta'_1 \eta'_2)\otimes d \eta''
+(-1)^{|\eta'|+|\eta''_1|}d \eta'\otimes d( \eta''_1 \eta''_2)\nonumber\\
&=&-\frac{1}{2}\sum\Big((-1)^{|\eta'_1|} (d \eta'_1)
\bar\eta'_2\otimes d \eta''-\bar\eta'_1(d \eta'_2) \otimes d \eta''\nonumber\\
&&+(-1)^{|\eta'|+|\eta''_1|}d\eta'\otimes (d\bar\eta''_1)\eta''_2-(-1)^{|\eta'|}d\bar\eta'\otimes\bar
\eta''_1(d\eta''_2)\Big).\label{eq:twistedcyclicity}
\end{eqnarray}
In the right hand side of the last equality, the first and the last summands cancel with each other,
due to the co-associativity of $A^{\textup{!`}}$. For 
the second and third summands, let us recall that in $A^!$, we have
$$\langle a, b\cdot c\rangle=(-1)^{(|b|+|c|)|a|}
\langle b\cdot c, \sigma^*(a)\rangle
=(-1)^{(|b|+|c|)|a|}
\langle b, c\cdot \sigma^*(a)\rangle,\quad \mbox{for all}\, a, b, c\in A^!.$$
This means on the linear dual space $A^{\textup{!`}}$,
suppose $(id\otimes \Delta)\circ\Delta(\eta)=\sum\eta'\otimes\eta''\otimes\eta'''$,
then
$$
\sum\eta'\otimes\eta''\otimes\eta'''=\sum(-1)^{|\eta'|(|\eta''|+|\eta'''|)}
\eta''\otimes\eta'''\otimes\sigma(\eta').
$$
Thus when passing to the twisted quotient space,
the second and third summands cancel with each other.
This proves the lemma.
\end{proof}

\begin{proof}[Proof of Theorem \ref{thm:firstmain}]
By Definition \ref{def:twistedbisymplecticstru}
and Lemma \ref{thm:existenceoftwistedbisymplectic},
we only need to show $\bm\omega$ is $\sigma$-invariant and
\begin{equation}\label{eq:qisintwistedcase}
\bm\Psi\circ\bm\iota_{(-)}\bm\omega:\widehat{ \mathbf D\mathrm{er}}\, R[1]
\rightarrow\widehat{\mathbf{\Omega}}^1_\sigma R[-1][n]
\end{equation}
is a quasi-isomorphism.

In fact, since 
on $A^!$ the pairing satisfies $\langle a, b\rangle=(-1)^{|a||b|}\langle
\sigma(b),a\rangle=\langle\sigma(a),\sigma(b)\rangle$,
we get that $\Delta(\eta)=\sum\eta'\otimes\eta''$ and hence $\bm\omega$ 
is $\sigma$-invariant. Also, since
$$\widehat{\mathbf{\Omega}}^1_\sigma R[-1]\cong R\otimes A^{\textup{!`}}\otimes R
\quad\mbox{and}\quad
\widehat{ \mathbf D\mathrm{er}}\, R[1]\cong R\otimes A^!\otimes R,$$
the isomorphism
\eqref{eq:qisintwistedcase}
on the generators becomes
$$
\bm\iota_{(-)}\bm\omega: A^{!}\to A^{\textup{!`}}[n],\, 
u\mapsto \sum u(\eta')\cdot \eta''[n],
$$
which is exactly the isomorphism given by \eqref{eq:modulemap}.
Therefore the theorem follows.
\end{proof}

In other words, $\bm\omega$ given by \eqref{eq:bisymplecticform} is a (derived) twisted
bi-symplectic structure on $A$.

As a by-product, we get the following result due to Van den Bergh \cite{VdB98} in general
and Reyes, Rogalski and Zhang \cite{RRZ} in the case of AS-regular algebras, which is in fact
the main motivation of the current paper:

\begin{corollary}[Van den Bergh \cite{VdB98}; Reyes-Rogalski-Zhang \cite{RRZ}]\label{cor:VdBPD}
Suppose $A$ is AS-regular of dimension $n$.
Then
\begin{equation}\label{eq:VdBNCPD}
\mathrm{HH}^\bullet(A)\cong\mathrm{HH}_{n-\bullet}(A, A_\sigma),
\end{equation}
which coincides with Van den Bergh's noncommutative Poincar\'e duality \cite{VdB98}.
\end{corollary}

\begin{proof}
Follows from Theorem \ref{thm:firstmain} and Corollary \ref{cor:twistedsympimpliestwistedPD}.
\end{proof}

\begin{remark}
In the case of Calabi-Yau algebras,
de Thanhoffer de V\"olcsey and Van den Bergh proved in \cite[Theorem 5.5]{dTdVVdB}
that there exists a volume class $\eta\in\mathrm{HH}_n(A)$ such
that the noncommutative Poincar\'e duality
$$\mathrm{HH}^\bullet(A)\stackrel{\cong}\longrightarrow
\mathrm{HH}_{n-\bullet}(A)$$
is given by capping with the volume class $\eta$.
In the general AS-regular algebra case, this could not be true at first glance,
since there is no cap-product between the Hochschild cohomology
and the {\it twisted} Hochschild homology.
Nevertheless, Theorem \ref{thm:firstmain} says that after choosing appropriate
chain complexes,
there does exist a volume class 
such that capping with it gives
the isomorphism \eqref{eq:VdBNCPD}.
\end{remark}

\subsection{Examples}

In this subsection, we study two examples of AS-regular algebras
with nontrivial Nakayama automorphism, and give their
twisted bi-symplectic structure explicitly.

\begin{example}\label{example:2dimAS}
Let $A= k\langle x_1,\cdots, x_n\rangle/(f)$,
where 
$$
f=(x_1,\cdots, x_n)M(x_1,\cdots, x_n)^T, M\in \mathrm{GL}_n(k),
$$ 
and $(f)$ means the ideal generated by $f$.
Dubois-Violette proved in \cite{DV2} that
$A$ is a Koszul AS-regular algebra of dimension $2$. 
The Koszul dual algebra $A^!$ is graded Frobenius of global 
dimension $2$ with
$$
A^!=k\oplus A^!_1\oplus A^!_2,
$$
where $A^!_1$ is the $k$-linear dual space of 
$\mathrm{Span}_k\{x_1,\cdots, x_n\}$ 
with the corresponding dual basis $\{\tilde x_1,\cdots, \tilde x_n\}$ and 
$A^!_2=(A^!_1\otimes A^!_1)/\{\alpha\in V^*\otimes V^*|\alpha(f)=0\}$,
which is 1-dimensional, say, spanned by $\{z\}$. 
The Frobenius pair $\langle -,-\rangle$ on $A^!$ is given by
$$
\langle a,b\rangle=aMb^T,\quad a=a_1\tilde x_1+\cdots+a_n\tilde x_n,\quad b=b_1\tilde x_1+\cdots+b_n\tilde x_n.
$$
From $\langle a,b\rangle=\langle b,\sigma^*(a)\rangle$ we get
$$
\sigma^*(a)=(\tilde x_1,\cdots,\tilde x_n)M^{-1}M^T (a_1,\cdots, a_n)^T, \quad \sigma^*(z)=z.
$$

By Van den Bergh \cite[Theorem 9.2]{VdBExist}, the Nakayama automorphism 
${\sigma}$ of $A$ is
$$
{\sigma}=\epsilon^{2+1}(\sigma^{*T})^{-1},
$$
where $\epsilon(a)=(-1)^ma$ for $a\in A_m$.
On $A_1$, for $x=\sum k_i x_i$, $k_i\in k$, 
$$
{\sigma}(x)=-(x_1,\cdots, x_n)M^TM^{-1}(k_1,\cdots, k_n)^T.
$$

Next we give the explicit twisted bi-symplectic structure on $A$. 
The Koszul dual coalgebra $A^\textup{!`}$ has the form 
$$
A^\textup{!`}=k\oplus A^{\textup{!`}}_1\oplus A^{\textup{!`}}_2
$$
The volume form $\eta=(x_1,\cdots, x_n)M(x_1,\cdots, x_n)^T=\sum_{i,j=1}^nm_{ij}x_ix_j$,
and therefore the twisted bi-symplectic structure is given by
$$
\bm\omega=d (\mathbf 1)\otimes d \big(\sum_{i,j}m_{ij}x_ix_j\big)+\sum_{i,j}m_{ij}d (x_i)\otimes d (x_j) 
+ d\big(\sum_{i,j}m_{ij} x_ix_j\big)\otimes d (\mathbf 1).
$$
\end{example}

\begin{example}[Quantum affine space]
Let $Q=\left(\begin{array}{ccc}
q_{11}&\cdots & q_{1n}\\
\vdots&\ddots&\vdots\\
q_{n1}&\cdots & q_{nn}
\end{array}
\right)$ be an $n\times n$ matrix over $k$ with 
$q_{ii}=1, q_{ij}q_{ji}=1$, for $1\leq i,j\leq n$. 
Let $A=k\langle x_1,\cdots, x_n\rangle/(x_jx_i-q_{ij}x_ix_j)$. 
In \cite[Proposition 4.1]{LWW2}, 
Liu, Wang and Wu proved that $A$ is a twisted Calabi-Yau algebra with
the Nakayama automorphism $\sigma$ given by
$$
\sigma(x_i)=(\Pi_{j=1}^nq_{ji})x_i,\quad i=1,\cdots,n.
$$
The Koszul dual algebra $A^!$ of $A$ is
$$
A^!=k\langle \tilde x_1,\cdots,\tilde x_n\rangle/(q_{ij}\tilde x_j\tilde x_i+\tilde x_i\tilde x_j|1\leq i,j\leq n).
$$
The volume form of $R=\Omega(A^{\textup{!`}})$ 
is
$$
\sum_{\sigma=(i_1,i_2)\cdots (i_k, i_{k+1})\in S_n}
\mathrm{sgn}(\sigma) ( q_{i_1 i_2}\cdots q_{i_ki_{k+1}}) x_{\sigma(1)}\cdots x_{\sigma(n)},
$$
from which we can write down the twisted symplectic form explicitly. 
For example,
for $n=2$, we have $Q=\left(\begin{array}{cc}
1&q_{12}\\
q_{21}& 1
\end{array} \right)$, then
$$
\sigma(x_1)=q_{21}x_1,\quad \sigma(x_2)=q_{12}x_2.
$$
Similarly to Example \ref{example:2dimAS}, the twisted symplectic structure is 
\begin{eqnarray*}
\bm\omega&=&d (\mathbf{1})\otimes d (x_2x_1-q_{12}x_1x_2)+d (x_2)
\otimes d(x_1)\\
&&-q_{12}d (x_1)\otimes d (x_2)
+d (x_2x_1-q_{12}x_1x_2)\otimes d(\mathbf{1}).
\end{eqnarray*}
For $n>2$,
the method is the same, and we leave it to the interested reader.
\end{example}

\begin{remark}
In Theorem \ref{thm:firstmain}, we have assumed that the AS-regular algebra
is Koszul. In fact, by the same technique, the theorem holds for $N$-Koszul
AS-regular algebras, too, such as the {\it down-up algebra} introduced by Benkart and Roby
in \cite{BR}.
\end{remark}

\section{Derived representation schemes}\label{sect:DRep}

From now on we study the twisted symplectic
structure on the derived representation schemes
of a Koszul AS-regular algebra. Derived representation schemes
were first introduced by Berest, Khachatryan and Ramadoss in \cite{BKR}
(see also \cite{BFPRW1,BFPRW2,BFR} for further discussions
and applications).
The goal of this and the subsequent section is to prove
Theorem \ref{thm:secondmain}.
At the same time, these two sections
also improve some results presented in \cite{CE} if we 
take the algebra to be Calabi-Yau, and hence serves
as supplement to it. 

\subsection{Representation functor of DG algebras}

Let $A$ be a DG $k$-algebra.
Let $V$ be a chain complex of finite total dimension.
Consider the following representation functor
\begin{equation}\label{functor:representationofDGAs}
\mathrm{Rep}_V(A): \mathsf{CDGA}\to\mathsf{Sets},\quad
B\mapsto \mathrm{Hom}_{\mathsf{DGA}}(A, \underline{\mathrm{End}}\, V\otimes B),
\end{equation}
where $\underline{\mathrm{End}}\, V$ is the DG algebra
of endomorphisms of $V$. 
It is proved in \cite[Theorem 2.1]{BKR} that
this functor is representable. That is, there exists a DG commutative algebra $A_V$,
which only depends on $A$ and $V$, such that
$$
\mathrm{Hom}_{\mathsf{DGA}}(A, \mathrm{End}\, V\otimes B)
=\mathrm{Hom}_{\mathsf{CDGA}}(A_V, B).
$$
In other words, we may view $\mathrm{Rep}_V(A)$ as a DG affine scheme,
and $A_V=\mathscr O(\mathrm{Rep}_V(A))$.

The rough idea of Berest et al in \cite{BKR} is as follows.
Let
$\mathsf{DGA}_{\mathrm{End}(V)}$ be the category of
DG algebras under $\underline{\mathrm{End}}\, V$, that is,
the objects of $\mathsf{DGA}_{\mathrm{End}(V)}$
are DG algebra maps $\underline{\mathrm{End}}\, V\to A$
and the morphisms are commutative triangles
$$
\xymatrix{
&\underline{\mathrm{End}}\, V\ar[ld]\ar[rd]&\\
A\ar[rr]&&B.
}
$$
Then the following functor
$$\mathcal G: \mathsf{DGA}_k\to\mathsf{DGA}_{\mathrm{End}(V)},\quad
A\mapsto\underline{\mathrm{End}}\, V\otimes A$$
is in fact an equivalence of categories, whose
inverse
is
$$
\mathcal G^{-1}: \mathsf{DGA}_{\mathrm{End}(V)}\to\mathsf{DGA}_k,\quad
(\underline{\mathrm{End}}\, V\to A)\mapsto A^{\underline{\mathrm{End}}(V)},
$$
where $A^{\underline{\mathrm{End}}(V)}$ is the centralizer of $\underline{\mathrm{End}}\, V$ in $A$;
see \cite[Lemma 2.1]{BKR} for more details.

Now introduce the following two functors
\begin{eqnarray}
&&\sqrt[V]{-}: \mathsf{DGA}_k\to\mathsf{DGA}_k,\quad A\mapsto 
(\underline{\mathrm{End}}\, V\ast_k A)^{\underline{\mathrm{End}}\, V},\\
&&(-)_V: \mathsf{DGA}_k\to\mathsf{CDGA}_k,\quad A\mapsto(\sqrt[V]{A})_{\natural\natural},
\end{eqnarray}
where $\ast$ is the coproduct of two $k$-algebras, and $(-)_{\natural\natural}$ is the ``commutativization" of the algebra,
namely, taking the quotient of the algebra by the two-sided ideal generated by the commutators.
The following is proved in \cite[Proposition 2.1]{BKR}:

\begin{proposition}\label{prop:twoadjunctions}
For any $A, B\in\mathsf{DGA}_k$ and $C\in\mathsf{CDGA}_k$, there are natural bijections:
\begin{enumerate}
\item[$(1)$] $\mathrm{Hom}_{\mathsf{DGA}_k}(\sqrt[V]{A},B)
\cong\mathrm{Hom}_{\mathsf{DGA}_k}(A,\underline{\mathrm{End}}\, V\otimes B)$;
\item[$(2)$] $\mathrm{Hom}_{\mathsf{CDGA}_k}(A_V, C)
\cong\mathrm{Hom}_{\mathsf{DGA}_k}(A,\underline{\mathrm{End}}\, V\otimes C)$.
\end{enumerate}
\end{proposition}

The second statement of this proposition exactly says that $A_V$ represents the functor
\eqref{functor:representationofDGAs}.
Now fixing $V$, we get a functor
$$
(-)_V: \mathsf{DGA}_k\to\mathsf{CDGA}_k,\quad A\mapsto A_V.
$$
Then \cite[Theorem 2.2]{BKR} says that $(-)_V$ has
a right adjoint
$$\underline{\mathrm{End}}\, V\otimes-: \mathsf{CDGA}_k\to\mathsf{DGA}_k,
\quad B\mapsto \underline{\mathrm{End}}\, V\otimes B
$$
such that
\begin{equation}
\label{functors:QuillenpairofRep}
(-)_V: \mathsf{DGA}_k\rightleftarrows \mathsf{CDGA}_k:\underline{\mathrm{End}}\, V\otimes-
\end{equation}
form a Quillen pair, and hence can be lifted
to their homotopy categories.

\subsection{Quillen's adjunction theorem}

We recommend the survey of Dwyer and Spalinski \cite{DS} for
an introduction to model categories.

Suppose $\mathcal A$ is a model category, then the homotopy category $\mathsf{Ho}(\mathcal A)$
of $\mathcal A$
is the category where the objects remain the same as those in $\mathcal A$,
and the morphisms for two objects, say $A$ and $B$,
are given by
\begin{equation*}
\mathrm{Hom}_{\mathsf{Ho}(\mathcal A)}(A,B)
:=\mathrm{Hom}_{\mathcal A}(QA,QB)/\mbox{quasi-isomorphisms},
\end{equation*}
where $QA$ and $QB$ are the {\it cofibrant resolutions} of $A$ and $B$ respectively
(for more details see \cite[\S5]{DS}).
Both $\mathsf{DGA}_k$ and $\mathsf{CDGA}_k$ are model categories,
and their homotopy categories are denoted by
$\mathsf{Ho}(\mathsf{DGA}_k)$ and $\mathsf{Ho}(\mathsf{CDGA}_k)$ respectively.

Now suppose $\mathcal A, B$ are two model categories, and two functors
$$F: \mathcal A\rightleftarrows \mathcal B: G$$
form a Quillen pair. Then Quillen's adjunction theorem says that 
the derived functors
$$\bm{L}F: \mathsf{Ho}(\mathcal A)\rightleftarrows \mathsf{Ho}(\mathcal B): \bm RG$$
exist, and again form a Quillen pair.
The functor $\bm LF$ is given by $\bm LF(A)=\gamma F(QA)$,
where $QA\to A$ is a(ny) cofibrant resolution of $A$,
and $\gamma: \mathcal A\to\mathsf{Ho}(\mathcal A)$
is the natural functor from $\mathcal A$ to its homotopy category,
which is the identity on objects,
and sends morphisms to the homotopy equivalence classes
of their liftings on cofibrant resolutions.

\subsection{Derived representation schemes}

Since the functors \eqref{functors:QuillenpairofRep}
form a Quillen pair, from the above argument 
they have the associated derived functors.

\begin{definition}[Berest, Khachatryan and Ramadoss \cite{BKR}]
For a chain complex $V$ of finite total dimension, the derived functor
$$\bm{L}(-)_V: 
\mathsf{Ho}(\mathsf{DGA})\to\mathsf{Ho}(\mathsf{CDGA}), A\mapsto (QA)_V$$
is called the {\it derived representation functor} in $V$.
For $A\in\mathsf{Ho}(\mathsf{DGA}_k)$, the DG affine scheme $\mathsf{Spec}(\bm L(A)_V)$ of $\bm L(A)_V$ is called
the {\it derived representation scheme}\footnote{In \cite{BKR}, the {\it derived representation
scheme} of $A$ is simply the DG commutative algebra $\bm L(A)_V$,
and is also denoted by $\mathrm{DRep}_V(A)$. Here we reserve this terminology
to mean its affine scheme, which is consistent with Yeung in 
\cite{Yeung}.} of $A$ in $V$, and is also denoted by $\mathrm{DRep}_V(A)$.
The homology $\mathrm H_\bullet(\bm L(QA)_V)$ is called the {\it representation homology} of $A$.
\end{definition}

The interested reader may also refer to Ciocan-Fontanine and Kapranov \cite{CFK}
for more details about derived schemes.

Suppose $A$ is an associative algebra, viewed as a DG algebra with zero differential,
then any DG free resolution of $A$ is a cofibrant resolution.
If we view the cofibrant resolution of $A$ as
the best approximation of $A$ by
DG free algebras, then $\mathrm{DRep}_V(-)$
best approximates $\mathrm{Rep}_V(-)$. In particular,
the zeroth homology
$$\mathrm H_0(\mathrm{DRep}_V(A))=\mathscr O(\mathrm{Rep}_V(A))$$
(see \cite[Theorem 2.5]{BKR}).

\begin{example}\label{ex:DRepofquasifree}
Assume $V=k^d$ and $A$ is an object in $\mathsf{Ho}(\mathsf{DGA}_k)$.
Suppose
$(R, \partial)=(k\langle x^{\alpha}\rangle_{\alpha\in\mathcal I},\partial)$
is a cofibrant resolution of $A$.
We construct a DG commutative algebra as follows:
to each $x^\alpha$ are associated variables
$x^\alpha_{i j}$, $1\le i, j\le d$, with the same grading
as $x^\alpha$. Consider the assignments
$x^\alpha\mapsto (x^\alpha_{i j})$, then the differential on these variables
are assigned 
such that
$$
\partial (x^\alpha_{i j})=(\partial x^\alpha)_{i j}.$$
$\mathrm{DRep}_V(A)$
is an object in $\mathsf{Ho}(\mathsf{CDGA})$
isomorphic to 
$(k[x^{\alpha}_{i j}]|_{\alpha\in\mathcal I},\partial)$.
\end{example}

From now on, we always assume $V$ is a vector space, that is, $V=k^d$ for some $n\in\mathbb N$.
$\mathrm{DRep}_V(A)$ is sometimes also denoted by $\mathrm{DRep}_d(A)$.

\subsection{The $\mathrm{GL}$-invariant subfunctor}

Observe that the general linear group $\mathrm{GL}(V)$, as a subspace of $\mathrm{End}\, V$,
acts from the right on the latter by conjugation $\alpha\mapsto g^{-1}\alpha g$, for all $g\in\mathrm{GL}(V)$.
It induces a right action on the functor $\mathrm{End}\, V\otimes-: \mathsf{CDGA}_k\to\mathsf{DGA}_k$.
Through the adjunction in Proposition \ref{prop:twoadjunctions}(2)
we get an action of $\mathrm{GL}(V)$ from the left
on the representation functor $(-)_V:\mathsf{DGA}_k\to \mathsf{CDGA}_k$.
Passing to the homotopy category, we get a $\mathrm{GL}(V)$-action
on $\mathrm{DRep}_n(A)$.

Thus we may consider the invariant subfunctor
\begin{equation}\label{functor:invariant}
(-)_V^{\mathrm{GL}}:\mathsf{DGA}_k\to\mathsf{CDGA}_k,\quad A\mapsto A_V^{\mathrm{GL}}.
\end{equation}
It is showed in \cite{BKR} that such a functor does not seem to have a right adjoint, and therefore
the Quillen Adjunction Theorem does not apply. Nevertheless, the authors showed
the following:

\begin{theorem}[\cite{BKR} Theoremn 2.6]
$(1)$ The functor \eqref{functor:invariant} has a total left derived functor
$$\bm L(-)^{\mathrm{GL}}_V:\mathsf{Ho}(\mathsf{DGA}_k)\to\mathsf{Ho}(\mathsf{CDGA}_k).$$

$(2)$ For any $A$ in $\mathsf{DGA}_k$, there is a natural isomorphism of graded algebras
$$\mathrm H_\bullet(\bm L(A)_V^{\mathrm{GL}})
\cong\mathrm H_\bullet(\bm L(A)_V)^{\mathrm{GL}}.$$
\end{theorem}

\subsection{Van den Bergh's functor}\label{sect:VdB}


For a DG algebra $A$, in this subsection
we study the differential forms
on $\mathrm{DRep}_V(A)$ and on the derived
quotient stack 
$[\mathrm{DRep}_V(A)/\mathrm{GL}(V)]$
in the sense of To\"en-Vezzosi \cite{TVII},
which
we denote by $\mathpzc{DRep}_V(A)$ 
and call the {\it derived moduli stack of representations} of $A$ in $V$.
We shall not go to the general theory of derived stacks; the interested reader
may refer to \cite{TVII}. However,
in what follows we shall give enough details which is
sufficient for our purpose.
The materials  
are taken from \cite{BFR,BKR,Yeung,Yeung2}.

\subsubsection{De Rham model for $\mathpzc{DRep}_V(A)$}

For derived quotient stacks, To\"en proposed in \cite[\S5]{Toen} the following model
for their differential forms:
Let $G$ be a reductive smooth group scheme over $k$ acting on
an affine derived scheme $Y=\mathsf{Spec}\, A$ and let $X:=[Y/G]$ be the derived quotient stack.
The 1-forms of $Y$, also called the cotangent complex of $Y$,
can be described algebraically (see also \S\ref{subsection:twistedforms}
 for some details). 
The cotangent complex $\mathbb L_{X}$ of $X$, 
pulled back to $Y$, is the fiber $\mathbb L$
of the natural morphism
$
\rho: \mathbb{L}_{Y}\to\mathscr{O}_Y\otimes_k\mathfrak{g}^{\vee}
$,
dual to the infinitesimal action of
$G$ on $Y$. In other words, the fiber $\mathbb L$ is quasi-isomorphic to the complex:
\begin{equation}\label{def:relativecotangent}
\mathbb L\cong\mathrm{cone} \{\mathbb{L}_{Y}
\stackrel{\rho}\to \mathscr{O}_Y\otimes_k\mathfrak{g}^{\vee}\}[-1].
\end{equation}
The group $G$ acts on $Y$ and on the morphism above, and hence on $\mathbb L$.
The complex of 1-forms on $X$ is thus given by 
$\mathbb L^G$. In general, the complex of $p$-forms $\Omega^p(X)$ is described as
\begin{equation}\label{eq:GLequivariantforms}
\Omega^p(X)\simeq \Big(\oplus_{i+j=p}(\wedge^i_A\mathbb L_Y)
\otimes_k\mathrm{Sym}_k^j(\mathfrak g^{\vee})[-j]\Big)^G.
\end{equation}

With To\"en's model described above, 
a model for the de Rham algebra of $\mathpzc{DRep}_V(A)$
is given as follows
(see also Yeung \cite{Yeung,Yeung2} for more details).
Let $R$ be a cofibrant resolution of $A$. Then
the 1-forms of $\mathpzc{DRep}_V(A)$
is the $\mathrm{GL}(V)$-invariant of
\begin{equation}\label{def:mathbbL}
\mathbb L=\mathrm{cone}\{\mathbb L_{\mathrm{DRep}_V(R)}\stackrel{\rho}\to
R_V\otimes_k\mathfrak{gl}(V)^\vee\}[-1].
\end{equation}
We now describe the map $\rho$ in \eqref{def:mathbbL}.
To this end, we go back to see the action
$$\mathrm{GL}(V)\times\mathsf{Spec}\, R_V\to\mathsf{Spec}\,R_V.$$
Since $\mathsf{Spec}\, R_V=\mathrm{Hom}_{\mathsf{DGA}}(R, \mathrm{End}\, V)$,
this action is given by
$$
\mathrm{GL}(V)\times\mathrm{Hom}_{\mathsf{DGA}}(R, \mathrm{End}\, V)\to
\mathrm{Hom}_{\mathsf{DGA}}(R,\mathrm{End}\, V),\quad
(g,x)\mapsto g\circ x\circ g^{-1}.
$$
Infinitesimally, the corresponding Lie action is
\begin{equation}\label{eq:infinitesimalLieaction}
\mathfrak{gl}(V)
\times
\mathrm{Hom}_{\mathsf{DGA}}(R, \mathrm{End}\, V)\to
\mathrm{Hom}_{\mathsf{DGA}}(R, \mathrm{End}\, V),\quad 
(u,x)\mapsto\{u\circ x-x\circ u\},
\end{equation}
where $\circ$ is the matrix multiplication.
It gives a Lie algebra map $\mathfrak{gl}(V)\to\mathfrak X(R_V):=\mathrm{Der}(R_V)$ as follows.
Suppose $x^\alpha_{ij}\in R_V$ (see Example \ref{ex:DRepofquasifree} for the notation), 
which corresponds to a representation
of $R$ in $V$, namely, $x: R\to\mathrm{End}\, V$.
Then according to 
\eqref{eq:infinitesimalLieaction}
we have a map
\begin{equation}\label{eq:infinLiealgebraonelements}
\mathfrak{gl}(V)\to\mathfrak X(R_V),\quad
 u\mapsto\left\{x^\alpha_{ij}\mapsto(u\circ x^\alpha-x^\alpha\circ u)_{ij}\right\}.
 \end{equation}
Dually, we get a map
$$\phi:\Omega^1(R_V) \to \mathfrak{gl}(V)^\vee\otimes R_V$$
which is given by
\begin{equation}\label{eq:dualofLieaction}
\phi(dx^\alpha_{ij})(u)=u\circ x^\alpha_{ij}\stackrel{\eqref{eq:infinLiealgebraonelements}}
=\sum_k u_{ik}x^\alpha_{kj}-x^{\alpha}_{ik}u_{kj}.
\end{equation}
If we identify $\mathfrak{gl}(V)^\vee$ with $\mathfrak{gl}(V)$
via the canonical pairing $\langle u, v\rangle=\mathrm{trace}(uv)$, which we denote by $\mathrm{Tr}^\vee$,
then we obtain from \eqref{eq:dualofLieaction} that
$$
\rho:=\mathrm{Tr}^\vee\circ\phi:  \Omega^1(R_V)  \to \mathfrak{gl}(V)\otimes R_V
$$
is given by
\begin{equation}\label{def:rho}
\rho: dx^\alpha_{ij} \mapsto  
\left(
\begin{array}{ccccc}
0&\cdots&x^\alpha_{1i}&\cdots&0\\
0&\cdots&x^\alpha_{2i}&\cdots&0\\
\cdots&\cdots&\cdots&\cdots&\cdots\\
0&\cdots&x^\alpha_{ni}&\cdots&0
\end{array}
\right)-\left(\begin{array}{cccc}
0&0&\cdots&0\\
\cdots&\cdots&\cdots&\cdots\\
x^\alpha_{j1}&x^\alpha_{j2}&\cdots&x^\alpha_{jn}\\
\cdots&\cdots&\cdots&\cdots\\
0&0&\cdots&0
\end{array}\right),
\end{equation}
where in the right of the arrow,
the first matrix has zero entries except the $j$-th column
and the second one has zero entries except the $i$-th row.
Thus by combining \eqref{def:relativecotangent}--\eqref{def:rho}, we obtain the following (see \cite{Toen,Yeung2}):

\begin{proposition}\label{prop:descriptionof1forms}
With $\mathbb L$ given by \eqref{def:mathbbL} and $\rho$ given by
\eqref{def:rho},
the 1-forms of $\mathpzc{DRep}_V(A)$
is $\mathbb L^{\mathrm{GL}(V)}$, and in general,
$$
\Omega^p(\mathpzc{DRep}_V(A))
\cong (\mathrm{Sym}^p\mathbb L)^{\mathrm{GL}(V)},
$$
where $\mathrm{Sym}^\bullet(-)$ means the graded symmetric product.
\end{proposition}

\subsubsection{Van den Bergh's functor}

Suppose $A$ is a DG algebra and $V$ is a chain complex.
Van den Bergh introduced a functor
\begin{equation}\label{functor:VdB}
(-)_V^{\mathrm{ab}}: \mathsf{Bimod}\, A\to\mathsf{DGMod}\, A_V, \quad 
M\mapsto M\otimes_{A^e}(\underline{\mathrm{End}}\, V\otimes A_V).
\end{equation}
This functor has a derived version as follows: suppose 
$R\to A$ is a cofibrant resolution of $A$.
Note that $R\otimes_A M$ is a DG $R$-bimodule, and let $F(R, M)$ be its
projective (cofibrant) resolution.
Berest et al proved the following:

\begin{proposition}[\cite{BKR} Corollary 5]\label{prop:VdBformodules}
Let $A\in\mathsf{Alg}$ and $M$ a complex of bimodules over $A$.
The assignment $M\mapsto F(R, M)_V$ induces a well defined functor,
$$
\bm L(-)^{\mathrm{ab}}_V: \mathscr D(\mathsf{Bimod}\, A)\to\mathscr D(\mathsf{DGMod}\, R_V),\quad
M\mapsto (F(R, M))_V^{\mathrm{ab}},
$$
which is independent of the choice of the resolutions $R\to A$ and $F(R, M)\to M$
up to equivalence of $\mathscr D(\mathsf{DGMod}\, R_V)$ inducing
the identity on homology.
\end{proposition}

In \cite{BFR,BKR}, $\bm L(-)^{\mathrm{ab}}_V$ is called the {\it derived Van den Bergh
functor}; sometimes it is also denoted
by $\bm L(-)_V$ if the content is clear.
Due to this proposition, in what follows we shall consider directly DG 
bimodules over $R$ instead of those
over $A$, which is then given by \eqref{functor:VdB} up to homotopy.

\subsubsection{Derived noncommutative tangent and cotangent complexes}

Van den Bergh \cite{VdBNCHam} as well as Berest et al \cite{BKR} proved the following:

\begin{lemma}\label{thm:VdBoftancotan}
Under Van den Bergh's functor,
$$
(\mathbf{\Omega}^1  R)_V=\Omega^1(R_V)\quad\mbox{and}\quad
( \mathbf D\mathrm{er}\, R)_V=\mathrm{Der}(R_V).
$$
\end{lemma}

\begin{proof}
See \cite[Proposition 3.3.4]{VdBNCHam} for the associative algebra case, and \cite[pages 661 and 671]{BKR} 
for the derived case.
\end{proof}

Now, it is direct to see
$$(R\otimes R)_V=(R\otimes R)\otimes_{R^e}(\mathrm{End}\, V\otimes R_V)=R_V\otimes \mathrm{End}\, V.$$
Combining this with Lemma \ref{thm:VdBoftancotan}
we obtain the following identity
\begin{equation}\label{eq:augmentedncform}
\big(\mathrm{cone}\{\mathbf{\Omega}^1 R\to R\otimes R\}\big)_V
=\mathrm{cone}\{\Omega^1(R_V)\to R_V\otimes\mathrm{End}\, V\}.
\end{equation}

\begin{lemma}\label{claim:idoftwocotangentmodels}
We have the following isomorphism:
\begin{equation}\label{eq:idoftwocotangentmodels}
\mathrm{cone}\{\Omega^1(R_V)\to R_V\otimes\mathrm{End}\, V\}
\cong
\mathrm{cone}\{\mathbb L_{\mathrm{DRep}_V(A)}\stackrel{\rho}\to R_V\otimes_k\mathfrak{gl}(V)^\vee\}.
\end{equation}
\end{lemma}

\begin{proof}
We identify $\mathfrak{gl}(V)^{\vee}$ with $\mathfrak{gl}(V)$
and notice that
$$R_V\otimes\mathfrak{gl}(V)\cong R_V\otimes \mathrm{End}\, V.$$
Therefore, the corresponding components of \eqref{eq:idoftwocotangentmodels} on both sides are isomorphic,
so we only need to check that the maps between two cones are the same.

To this end, let us first describe the map in 
$$\mathrm{cone}\{\Omega^1(R_V)\to R_V\otimes\mathrm{End}\, V\},$$
which we denote by $\kappa$.
By definition, we have the following diagram
$$\xymatrixcolsep{4pc}
\xymatrix{
R\otimes A^{\textup{!`}}\otimes R\ar@{^{(}->}[r]^{i}\ar@{~>}[d]^{(-)_V}&R\otimes R\ar@{~>}[d]^{(-)_V}\\
\Omega^1(R_V)\ar[r]^{\kappa}&\mathrm{End}\, V\otimes R_V,
}
$$
where $i$ is the inclusion.
That is, $\kappa=i_V$. From this we have an explicit expression for $\kappa$:
$$\kappa\cong i\otimes id: (R\otimes A^{\textup{!`}}\otimes R)\otimes_{R^e} (\mathrm{End}\, E\otimes R_V)\to
(R\otimes R)\otimes_{R^e}(\mathrm{End}\, E\otimes R_V).$$
Pick an element $dx^\alpha\otimes e_{ij}\in
(R\otimes A^{\textup{!`}}\otimes R)\otimes_{R^e} (\mathrm{End}\, E\otimes R_V)$,
which is identified with $dx^\alpha_{ij}$,
then
\begin{eqnarray}\label{eq:theimageofkappa}
\kappa(dx^{\alpha}_{ij})&=&(x^\alpha\otimes 1-1\otimes x^\alpha)\otimes e_{ij}\nonumber\\
&=&x^{\alpha}\circ e_{ij}-e_{ij}\circ x^{\alpha}\nonumber\\
&=&\left(
\begin{array}{ccccc}
0&\cdots&x^\alpha_{1i}&\cdots&0\\
0&\cdots&x^\alpha_{2i}&\cdots&0\\
\cdots&\cdots&\cdots&\cdots&\cdots\\
0&\cdots&x^\alpha_{di}&\cdots&0
\end{array}
\right)-\left(\begin{array}{cccc}
0&0&\cdots&0\\
\cdots&\cdots&\cdots&\cdots\\
x^\alpha_{j1}&x^\alpha_{j2}&\cdots&x^\alpha_{jd}\\
\cdots&\cdots&\cdots&\cdots\\
0&0&\cdots&0
\end{array}\right),
\end{eqnarray}
where in the second equality,
$\circ$ means the action (representation) of $x^\alpha$
on the corresponding matrices.
Now observe that \eqref{eq:theimageofkappa}
is exactly the same as $\rho$ given by \eqref{def:rho}, from which the lemma follows.
\end{proof}

\begin{theorem}[See also Yeung \cite{Yeung}]\label{thm:tangentandcotangentonDRep}
Suppose $A$ is a DG algebra with a cofibrant resolution $R\to A$.
Then
$$\mathfrak X(\mathpzc{DRep}_V(A))[1]\cong
\big((\widehat{\mathbf{D}\mathrm{er}}\, R[1])_V\big)^{\mathrm{GL}(V)}
$$
and
$$
\Omega^1(\mathpzc{DRep}_V(A))[-1]
=\big((\widehat{\mathbf{\Omega}}^1  R[-1])_V\big)^{\mathrm{GL}(V)}.
$$
\end{theorem}

\begin{proof}
By Proposition \ref{prop:descriptionof1forms},
we know that
$$\Omega^1(\mathpzc{DRep}_V(A))
=\big(\mathrm{cone}\{\Omega^1(R_V)\to R_V\otimes_k\mathfrak{gl}(V)^{\vee}\}\big)^{\mathrm{GL}(V)}$$
while
\begin{eqnarray*}
\mathrm{cone}\{\Omega^1(R_V)\to R_V\otimes_k\mathfrak{gl}(V)^{\vee}\}
&\stackrel{\eqref{eq:idoftwocotangentmodels}}
=&\mathrm{cone}\{\Omega^1(R_V)\to R_V\otimes\mathrm{End}\,V\}\\
&\stackrel{\eqref{eq:augmentedncform}}=&\mathrm{cone}\{\mathbf{\Omega}^1 R\to R\otimes R\}_V\\
&\stackrel{\eqref{eq:Yeungsnc1forms}}=&
(\widehat{\mathbf{\Omega}}^1  R)_V.
\end{eqnarray*}
This means
$$
\Omega^1(\mathpzc{DRep}_V(A))[-1]
=\big((\widehat{\mathbf{\Omega}}^1 R[-1])_V\big)^{\mathrm{GL}(V)}.
$$
Similarly,
$$
\mathfrak X(\mathpzc{DRep}_V(A))[1]
=\big((\widehat{ \mathbf D\mathrm{er}}\, R[1])_V\big)^{\mathrm{GL}(V)},
$$
whose proof is left to the reader.
This proves
the theorem.
\end{proof}

\subsubsection{The derived Procesi trace map}\label{subsubsect:derivedProcesi}

Prior to Van den Bergh, Procesi gave a map in \cite{Procesi}
$$
\mathrm{Tr}: A_\natural\to (A)_V^{\mathrm{GL}},\, \bar a\mapsto\{\rho\mapsto \mathrm{Trace}(\rho(a))\},
$$
where $\bar a$ means the equivalence class represented by $a\in A$.
This map was latter generalized by Berest et al in \cite{BFR,BKR} to $A$-bimodules as follows:
first, recall that Van den Bergh's functor \eqref{functor:VdB} is
$$(-)_V^{\mathrm{ab}}: \mathsf{Bimod}\, A\to\mathsf{DGMod}\, A_V, \, M\mapsto M_V^{\mathrm{ab}}
:= M\otimes_{A^e}(\mathrm{End}\, V\otimes A_V).$$
There is an analogue of Procesi's map for bimodules (see \cite[(5.11)]{BKR} or \cite[(87)]{BFR})
$$
\mathrm{Tr}_V(M): M\to  \mathrm{End}\, V\otimes M_V^{\mathrm{ab}}\to M_V^{\mathrm{ab}},
$$
which is a trace map, that is, it factors through $M_\natural:=M/[M, A]$.
It therefore gives a morphism of functors
$$\mathrm{Tr}_V: (-)_\natural\to (-)_V.$$
Now, the {\it derived} Procesi map is given by
\begin{equation}\label{def:derivedtracemap}
\mathrm{Tr}_V: \bm L(-)_\natural\to \bm L(-)_V.
\end{equation}
The difference between
$\bm L(M)_V^{\mathrm{ab}}$
and
$\mathrm{Tr}_V(\bm L(M))$
is that,
we would view the former as sheaves
on $\mathrm{DRep}_V(A)$
and the latter as their global sections.
We learned this point of view from Yeung \cite{Yeung,Yeung2}.

\begin{example}\label{traceofforms}
Suppose $R\to A$ is a DG free resolution of $A$,
then
$\widehat{\mathbf{\Omega}}(R)[-1]$ is a DG free $R$-bimodule.
By Lemma \ref{claim:idoftwocotangentmodels},
we have
$$
\Omega^1(\mathpzc{DRep}_V(A))[-1]
=\big((\widehat{\mathbf{\Omega}}^1 R[-1])_V\big)^{\mathrm{GL}(V)}.
$$
Recall that $(\widehat{\mathbf{\Omega}}^1R[-1])_\natural=
\widehat{\mathrm{DR}}^1(R)$.
Then by \eqref{def:derivedtracemap},
the morphism $\mathrm{Tr}_V$ is given by
$$
\mathrm{Tr}_V: \widehat{\mathrm{DR}}^1(R)\to \big((\widehat{\mathbf{\Omega}}^1 R[-1])_V\big)^{\mathrm{GL}(V)},\,
dx^\alpha \mapsto \mathrm{Tr}(dx^{\alpha}\otimes e_{ij})=\displaystyle\sum_{i=1}^d dx^{\alpha}_{ii}.
$$
Similarly, for 2-forms we have
$\mathrm{Tr}_V(dx^\alpha\otimes dx^\beta)=\displaystyle\sum_{i,j=1}^d dx^{\alpha}_{ij}\otimes dx^{\beta}_{ji}.$
\end{example}

\section{Derived representation schemes of AS-regular algebras}\label{sect:applications}

Now suppose $A$ is a Kosul AS-regular algebra of dimension $n$. 
In this section we study the twisted symplectic structure on
$\mathpzc{DRep}_V(A)$,
and prove Theorem \ref{thm:secondmain}.

\subsection{Twisted differential forms}\label{subsection:twistedforms}
The twisted differential forms of DG commutative
algebras are defined similarly to those of DG associative
algebras. Some other studies of twisted differential forms
can be found in, for example, \cite{Andre,KSA,LSQ}.

Let $(A, \mu, \partial)$ be DG commutative algebra over $k$. Recall that the set of K\"ahler 
differential forms of $A$ is
$\Omega^1 A:=I/I^2$,
where
$I=\ker\{\mu: A\otimes A\to A, x\otimes y\mapsto xy\}$.
Now suppose $A$ has
an automorphism $\sigma\in\mathrm{Aut}\, A$. Let 
\begin{eqnarray*}
J:=\ker\{\mu_\sigma: A\otimes A\to A,\, x\otimes y\mapsto x\cdot \sigma(y)\}.
\end{eqnarray*}
and let
$\Omega^1_\sigma A:=J/J^2$,
which is called the set of {\it twisted K\"ahler differential forms} of $A$.
We have the following:

\begin{lemma}\label{lemma:derivationandbimodule}
Let $\Omega^1_\sigma A$
be as above. Then
\begin{enumerate}
\item[$(1)$] for any $a\in A$, the map 
$$d_\sigma:A\to\Omega^1_\sigma A,\,  a\mapsto \sigma(a)\otimes 1-1\otimes a$$
satisfies
$d_\sigma(a b)=d_\sigma(a)\cdot b+\sigma(a)\cdot d_\sigma(b)$, for all $a, b\in A$;

\item[$(2)$] for any $u\in \Omega^1_\sigma A$ and $a\in A$, we have
$$
\sigma(a)\cdot u-(-1)^{|a||u|}u\cdot a=0\in \Omega^1_\sigma A.
$$
\end{enumerate}
\end{lemma}

\begin{proof}
(1) For any $a,b\in A$, we have
\begin{align*}
&d_\sigma(a)\cdot b+ \sigma(a)\cdot d_\sigma(b)\\
&= \big(\sigma(a)\otimes 1-1\otimes a\big)
\cdot b+\sigma(a)\cdot \big( \sigma(b)\otimes 1-1\otimes  b\big)\\
&= \big(\sigma(a)\otimes b-1\otimes ab\big)
+ \big(\sigma(ab)\otimes 1- \sigma(a)\otimes b\big)\\
&= \sigma(ab) \otimes 1-1\otimes ab\\
&=d_\sigma(ab).
\end{align*}

(2) We only need to show that, for any $u\in J/J^2$ and $a\in A$,
if $u$ is represented by
$\sum_i u_i\otimes v_i\in J$, then
$$\sigma(a)\cdot\big(\sum_i u_i\otimes v_i\big)-(-1)^{(|u_i|+|v_i|)|a|}
\big(\sum_i u_i\otimes v_i\big)\cdot a\in J^2.$$
In fact,
\begin{align*}
&\sigma(a)\cdot\big(\sum_i u_i\otimes v_i\big)-(-1)^{(|u_i|+|v_i|)|a|}
\big(\sum_i u_i\otimes v_i\big)\cdot a\\
&=\big(\sigma(a)\otimes 1\big)\cdot \big(\sum_i u_i\otimes v_i\big)
-\big(1\otimes a\big) \cdot\big(\sum_i u_i\otimes v_i\big)\\
&=\big(\sigma(a)\otimes 1-1\otimes a\big)\cdot
\big(\sum_i u_i\otimes v_i\big)\in J^2.\qedhere\end{align*}
\end{proof}

By this lemma, if we define the left and right action of $A$ on $\Omega^1_\sigma A$ by
$$
x\circ u\circ y:=\sigma(x)uy,\quad\mbox{for all}\, u\in\Omega^1_\sigma A\;\mbox{and}\; x, y\in A,
$$
then
$\Omega^1_\sigma A$ is a DG $A$-module.
Moreover, it is direct to see the following map
\begin{equation}\label{eq:defofpsi}
\Psi: \Omega^1 A\mapsto \Omega^1_\sigma A,\, dx
\mapsto d_\sigma(x) 
\end{equation}
is an isomorphism of $A$-modules.

Parallel to the noncommutative case,
in what follows we prefer to call 
$\Omega^1 A[-1]$ and $\Omega^1_\sigma A[-1]$
the sets of 1-forms and twisted 1-forms of $A$ respectively. 
In this case, we have $d\circ\partial+\partial\circ d=0$
and  $d_\sigma\circ\partial+\partial\circ d_\sigma=0$,
and moreover, $d:x\mapsto dx$ and $d_\sigma:x\mapsto d_\sigma x$
satisfies the Leibniz and twisted Leibniz rule respectively.

The {\it twisted differential forms} of $A$,
denoted by $\Omega^\bullet_\sigma A$, is given as follows (see \cite{KSA}):
let $T_A(\Omega^1 A[-1])$ be 
the tensor algebra generated by $\Omega^1 A[-1]$
over $A$, with $\partial$ and $d$
extending to it by derivation and by letting $d^2=0$,
then $\Omega^\bullet_\sigma A$ is the quotient of
$T_A(\Omega^1 A[-1])$  by the two-sided DG ideal
generated by $x\cdot y-(-1)^{|x||y|}y\cdot\sigma(x)$,
for all $x\in A$ and $y\in\Omega^1 A[-1]$, or both 
$x, y\in\Omega^1 A[-1]$. Observe that the degree one part
of $\Omega^\bullet_\sigma A$ is exactly $\Omega^1_\sigma A[-1]$
given above.

\subsection{Twisted symplectic structure}

Now suppose $\omega\in\Omega^2_\sigma A$, similarly to the noncommutative
case, we say it is $\sigma$-invariant
if $\sigma(\omega)=\omega$, where $\sigma$ is induced from $\sigma: A\to A$.
For such an $\omega$, suppose it is represented by $\sum da_1\otimes da_2$; then
we may define the following contraction
\begin{equation}\label{eq:defofphi}
\Psi\circ\iota_{(-)}\omega:\mathfrak X(A)\to\Omega^1_\sigma A[n-2],\, 
X\mapsto
\sum X(da_1)\Psi(da_2)
+X(da_2)\Psi(\sigma(da_1)),
\end{equation}
where $\mathfrak X(A)$
is identified with $\mathrm{Hom}_A(\Omega^1 A, A)$ and hence $\mathfrak X(da_1)$
is the evaluation, and $\Psi$ is given by \eqref{eq:defofpsi}.

\begin{lemma}
If $\omega$ is $\sigma$-invariant 
and extends to be closed, then \eqref{eq:defofphi} is a map of DG $A$-modules.
\end{lemma}

\begin{proof}
This is completely analogous to Lemma \ref{lemma:NCcontraction}.
\end{proof}

The following definition is completely parallel to Definition \ref{def:twistedbisymplecticstru};
if $\sigma$ is the identity map,  
then it coincides with the one of Pantev et. al. \cite{PTVV} in the DG affine case.

\begin{definition}[Twisted symplectic structure]
Let $A$ be a DG free commutative algebra with $\sigma\in\mathrm{Aut}\, A$.
A twisted 2-form $\omega\in\Omega^2_\sigma A$ of total degree $n$
is called a {\it twisted $(2-n)$-shifted
symplectic structure}
if 
\begin{enumerate}
\item it extends to be closed in the negative cyclic complex associated to
$\Omega^2_\sigma  A$;
\item it is $\sigma$-invariant and induces a quasi-isomorphism
$$
\Psi\circ\mathrm{\iota}_{(-)}\omega: \mathfrak X(A)[1]\to\Omega^1_\sigma  A[-1][n].
$$
of DG $A$-modules.
\end{enumerate}
If $A$ is a DG free resolution of a twisted DG algebra $B$, then we say $B$ has a {\it derived}
twisted symplectic structure if $A$ has a twisted symplectic structure.
\end{definition}

The above definition can be generalized to the global quotients
of derived affine schemes, viewed as a derived stack,
without any difficulty.

Now suppose $A$ is a Koszul AS-regular algebra of dimension $n$.
Let $R=\Omega(A^{\textup{!`}})$ be the cobar construction of
its Koszul dual coalgebra $A^{\textup{!`}}$.
Let $x^0={\bf 1}, x^1,\cdots, x^\ell$ be a basis of $A^{\textup{!`}}$.
By Example \ref{ex:DRepofquasifree}, for $V=k^d$,
$\mathrm{DRep}_V(A)$ is isomorphic to
$$
R_V=k[x^{\alpha}_{ij}|1\le \alpha\le \ell,1\le i,j\le d, |x^\alpha_{ij}|=|x^{\alpha}|-1].
$$

\begin{proof}[Proof of Theorem \ref{thm:secondmain}]
Denote 
$
\mathbb T=(\widehat{ \mathbf D\mathrm{er}}\, R[1])_V
$ and 
$
\mathbb L=(\widehat{\mathbf{\Omega}}^1 R[-1])_V$.
Then by Theorem \ref{thm:tangentandcotangentonDRep},
$\mathbb T^{\mathrm{GL}(V)}$ and $\mathbb L^{\mathrm{GL}(V)}$
are the tangent and cotangent complexes of
$\mathpzc{DRep}_V(A)$ respectively.
By the same argument, it is also direct to see that
the $\mathrm{GL}(V)$-invariant of
$\mathbb L_\sigma:=(\widehat{\mathbf{\Omega}}^1_\sigma R[-1])_V$,
denoted by $\mathbb L_\sigma^{\mathrm{GL}(V)}$,
is the twisted cotangent complex of $\mathpzc{DRep}_V(A)$.
Thus by applying Van den Bergh's functor
to Theorem \ref{thm:firstmain}, we get Theorem \ref{thm:secondmain}.
Since the computations are quite involved, we here give some details of the proof.

First of all,
recall that
in \eqref{eq:augmentedncform}--\eqref{eq:idoftwocotangentmodels}
we have identified
$$(R\otimes R)_V\cong
(R\otimes d\mathbf 1\otimes R)_V\cong
R_V\otimes\mathrm{End}\, V\cong R_V\otimes\mathfrak{gl}(V).
$$
As before, we write $\{x^0_{ij}\}$ to be the set of generators of $R_V\otimes\mathrm{End}\, V$
(the corresponding elementary matrices).
Let
$$
\omega:=\frac{1}{2}
\sum_{i,j=1}^d d \eta'_{ij}\otimes d \eta''_{ji},
$$
which is exactly the image of 
$\bm{\omega}=\frac{1}{2}\sum d \eta'\otimes d \eta''
\in\widehat{\mathrm{DR}}^2_\sigma R$ 
under
the derived Procesi map $\mathrm{Tr}_V$ described in 
Example \ref{traceofforms}.
It is straightforward to see that $\omega$ is: 1) $\sigma$-invariant of total degree $n$,
and 2) closed under both $d$ and $\partial$.

Now by \eqref{eq:defofphi} the contraction with $\omega$ is given by
\begin{equation}\label{eq:twistedqisonDRep}
\Psi\circ\iota_{(-)}\omega: \mathbb T\to \mathbb L_\sigma[n],\,
D{\tilde x^\alpha_{ij}}
\mapsto\sum_{\beta=0}^{\ell}\langle \tilde x^{\alpha},\tilde x^{\beta}\rangle d_\sigma(x^\beta_{ij}),
\end{equation}
which extends to a map of $R_V$-modules.
Here $
D\tilde x^{\alpha}_{ij}$ is $(\bm D\tilde x^{\alpha})_{ij}$, whose linear dual is $dx^{\alpha}_{ji}$.
We claim that
\eqref{eq:twistedqisonDRep}
is a quasi-isomorphism of chain complexes.
First,
denote $\delta: \mathbb T\to \mathbb T$ to be the
internal differential on $\mathbb T$
via
$\delta(D\tilde x^{\alpha}_{ij}):=(\delta(\bm D\tilde x^{\alpha}))_{ij}$,
where $\delta$ in the right hand side is given by
\eqref{eq:internaldiffonDer}.
More precisely, we have
\begin{equation*}
\delta(D\tilde x^{\alpha}_{ij})=-\sum_{\gamma}\sum_k 
(-1)^{|\alpha|+|\gamma|-1}D(\tilde x^{\gamma}\tilde x^{\alpha})_{ik}\cdot x^{\gamma}_{kj} 
-x^{\gamma}_{ik}\cdot D(\tilde x^{\alpha}\tilde x^{\gamma})_{kj}.
\end{equation*}
Here, for simplicity, we write $(-1)^{|x^\alpha|}$ as $(-1)^{|\alpha|}$.
Denote $\Phi:=\Psi\circ\iota_{(-)}\omega$; we have
\begin{align}
&\Phi\big(\delta(D{\tilde x^{\alpha}_{ij}})\big)\nonumber\\
&
{=}\Phi\big(-\sum_{\gamma}\sum_k 
(-1)^{|\alpha|+|\gamma|-1}D(\tilde x^{\gamma}\tilde x^{\alpha})_{ik}\cdot x^{\gamma}_{kj} 
-x^{\gamma}_{ik}\cdot D(\tilde x^{\alpha}\tilde x^{\gamma})_{kj}\big)\nonumber\\
&
{=}
-\sum_{\gamma}\sum_{\beta} \sum_{k}
(-1)^{|\alpha||\gamma|+|\beta|(|\gamma|-1)}
\langle \tilde x^{\gamma} \tilde x^{\alpha},\tilde x^{\beta}\rangle
d_\sigma({x^{\beta}_{ik}})\cdot\sigma(x^{\gamma}_{kj}) -\langle \tilde x^{\alpha}\tilde x^{\gamma},\tilde x^{\beta}\rangle 
x^{\gamma}_{ik}\cdot d_\sigma({x^{\beta}_{kj}})\nonumber\\
&=-\sum_{\gamma}\sum_{\beta} \sum_{k}(-1)^{|\beta|}
\langle\tilde x^{\alpha}, \tilde x^{\beta}\sigma^*(\tilde x^{\gamma})\rangle
d_\sigma({x^{\beta}_{ik}})\cdot\sigma(x^{\gamma}_{kj})
-\langle \tilde x^{\alpha}, \tilde x^{\gamma}\tilde x^{\beta})\rangle x^{\gamma}_{ik}\cdot 
d_\sigma({x^{\beta}_{kj}})\nonumber\\
&= -\sum_{\zeta}\sum_{k}\langle \tilde x^{\alpha},\tilde x^{\zeta}\rangle\big( 
(-1)^{|{\zeta'}|} d_\sigma({x^{\zeta'}_{ik}})\cdot  \sigma(x^{\zeta''}_{kj})
-x^{\zeta'}_{ik}\cdot d_\sigma({x^{\zeta''}_{kj}})\big)\nonumber\\
&=- \sum_{\zeta}\sum_{k} 
\langle \tilde x^{\alpha}, \tilde x^{\zeta}\rangle\big((-1)^{|{\zeta'}||{\zeta''}|}x^{\zeta''}_{kj} \cdot d_\sigma({x^{\zeta'}_{ik}})
-x^{\zeta'}_{ik}\cdot d_\sigma({x^{\zeta''}_{kj}})\big),\label{eq:Xipartial}
\end{align}
while
\begin{eqnarray}
\partial(\Phi(D{\tilde x^{\alpha}_{ij}}))&=&\sum_{\beta}\langle \tilde x^\alpha, \tilde x^\beta\rangle\cdot
\partial (d_\sigma(x^\beta_{ij}))=-\sum_{\beta}\langle\tilde x^\alpha, \tilde x^\beta\rangle\cdot d_\sigma
\circ\partial(x^\beta_{ij})\nonumber\\
&=&-\sum_{\beta}\langle \tilde x^\alpha,\tilde x^\beta\rangle\cdot d_\sigma\big(\sum_k (-1)^{|{\beta'}|}
x^{\beta'}_{ik}x^{\beta''}_{kj}\big)\nonumber\\
&=&-\sum_{\beta}
\sum_k\langle \tilde x^\alpha, \tilde x^\beta\rangle \big(   
(-1)^{|{\beta'}||{\beta''}|}x^{\beta''}_{kj} d_\sigma(x^{\beta'}_{ik})
-x^{\beta'}_{ik}d_\sigma(x^{\beta''}_{kj})\big).
\label{eq:partialXi}
\end{eqnarray}
Observe that \eqref{eq:Xipartial} and \eqref{eq:partialXi} are equal to each other,
from which we obtain that $\Phi$ is a chain map.

We now show $\Phi$ is non-degenerate, that is, it is a quasi-isomorphism of DG $R_V$-modules.
In fact, this is a direct corollary of the non-degeneracy of the pairing on $A^{\textup{!`}}$
and the non-degeneracy of the canonical pairing on matrices (namely, the trace of the matrix product).

Taking the $\mathrm{GL}$-invariants
of \eqref{eq:twistedqisonDRep} gives the desired quasi-isomorphism from the tangent complex
to the twisted cotangent complex on $\mathpzc{DRep}_V(A)$.
This means $\omega$ given above is a twisted symplectic form, from
which the theorem follows.
\end{proof}


\begin{thebibliography}{100}

\bibitem{Andre}Y. Andr\'e,
Diff\'erentielles non commutatives et th\'eorie de Galois diff\'erentielle ou aux
diff\'erences, \emph{Ann. Sci. \'Ecole Norm. Sup.} 4 ({\bf 34}):5 (2001), 685--739.


\bibitem{AS}
M. Artin and W.F. Schelter,
Graded algebras of global dimension 3,
\emph{Adv. Math.} \textbf{66} (1987), 171--216.


\bibitem{BFPRW1}Y. Berest, G. Felder, S. Patotski, A. Ramadoss and T. Willwacher,
Representation homology, Lie algebra cohomology and the derived Harish-Chandra homomorphism,
\emph{J. Eur. Math. Soc.} \textbf{19} (2017), no. 9, 2811--2893.

\bibitem{BFPRW2}Y. Berest, G. Felder, S. Patotski, A. Ramadoss and T. Willwacher,
Chern-Simons forms and higher character maps of Lie representations,
\emph{Int. Math. Res. Not. IMRN} (2017) no. {\bf 1}, 158--212.

\bibitem{BFR}
Y. Berest, G. Felder and A. Ramadoss,
Derived representation schemes and noncommutative geometry,
\emph{Contemp. Math.} \textbf{607} (2014), 113--162.

\bibitem{BKR}
Y. Berest, G. Khachatryan and A. Ramadoss,
Derived representation schemes and cyclic homology,
\emph{Adv. Math.} \textbf {245} (2013), 625--689.


\bibitem{BR}
G. Benkart and T. Roby,
Down-up algebras,
\emph{J. Algebra} \textbf{206} (1998), 305--344.


\bibitem{BB}
A. Berglund and K. B\"orjeson,
Koszul $A_\infty$-algebras and free loop space homology,
\emph{Proceedings of the Edinburgh Mathematical Society}, \textbf{63} (2020), no 1, 37--65.



\bibitem{BZ}
K.A. Brown and J.J. Zhang,
Dualizing complexes and twisted Hochschild (co)homology for Noetherian Hopf algebras,
\emph{J. Algebra} \textbf{320} (2008), 1814--1850.


\bibitem{CE}
X. Chen and F. Eshmatov,
Calabi-Yau algebras and the shifted non-commutative symplectic structure,
\emph{Adv. Math.} \textbf{367} (2020), 107126.

\bibitem{CFK}I. Ciocan-Fontanine and M. Kapranov,
Derived Quot schemes, \emph{Ann. Sci. ENS} \textbf{34} (2011) 403-440.

\bibitem{CB}
W. Crawley-Boevey, 
Poisson structures on moduli spaces of representations,
\emph{J. Algebra} \textbf{325} (2011), 205--215. 



\bibitem{CBEG}
W. Crawley-Boevey, P. Etingof and V. Ginzburg, 
Noncommutative geometry and quiver algebras,
\emph{Adv. Math.} \textbf{209}(2007), no. 1, 274--336. 

\bibitem{dTdVVdB}
L. de Thanhoffer de V\"olcsey and M. Van den Bergh,
Calabi-Yau deformations and negative cyclic homology. 
\emph{J. Noncommut. Geom.} {\bf 12} (2018), no. 4, 1255--1291. 


\bibitem{DV2}
M. Dubois-Violette,
Multilinear forms and graded algebras, \emph{J. Algebra} \textbf{317} (2007), 198--225.



\bibitem{DS}
W.G. Dwyer and J. Spalinski,
Homotopy theories and model categories,
\emph{Handbook of algebraic topology}, 73--126, North-Holland, Amsterdam, 1995. 

\bibitem{Ginzburg}
V. Ginzburg, Calabi-Yau algebras,
arXiv: 0612139v3.

\bibitem{HK}
T. Hadfield and U. Kr\"ahmer,
Twisted homology of the quantum $SL(2)$,
\emph{K-Theory}, \textbf{34}, 327--360.




\bibitem{KSA}
M. Karoubi and M. Suarez-Alvarez, 
Twisted K\"ahler differential forms,
\emph{J. Pure Appl. Algebra} \textbf{118} (2003), 279--289.



\bibitem{KR}
M. Kontsevich and A. Rosenberg, Noncommutative smooth spaces. 
\emph{The Gelfand Mathematical Seminars 1996--1999}, 85--108, 
Birkh\"auser Boston, Boston, MA, (2000). 

\bibitem{KMT}
J. Kustermans, G. Murphy and L. Tuset, Differential calculi over 
quantum groups and twisted cyclic cocycles, 
\emph{J. Geom. Phys.} \textbf{44} (2003), no. 4  570--594 .

\bibitem{LSQ}
E. Le Stum and A. Quir\'os,
Formal confluence of quantum differential operators,
\emph{Pacific J. Math.}, \textbf{292} (2018) no. 2, 427--478.

\bibitem{Liu}
L. Liu, 
Koszul duality and the Hochschild cohomology of Artin-Schelter algebras,
\emph{Homology, Homotopy and Appplications} \textbf{22} (2020), no.2, 181--202.



\bibitem{LWW2}
L.Y. Liu, S.Q. Wang and Q.S. Wu,
Twisted Calabi-Yau property of Ore extensions,
\emph{J. Noncommut. Geom.} \textbf{8} (2014), 587--609.


\bibitem{LV}
J.-L. Loday and B. Vallette,
\emph{Algebraic operads},
{Grundlehren der mathematischen Wissenschaften} {\bf 346}. Springer, Heidelberg, 2012. 



\bibitem{PTVV}
T. Pantev, B. To\"en, M. Vaqui\'e and G. Vezzosi,
Shifted symplectic structures,
\emph{Publ. Math. Inst. Hautes \'Etudes Sci.} \textbf{117} (2013), 271--328. 

\bibitem{Pridham}
J.P. Pridham,
Non-commutative derived moduli prestacks,
arXiv:2008.11684.

\bibitem{Procesi}
C. Procesi, 
The invariant theory of $n\times n$ matrices,
\emph{Adv. Math.} \textbf{19} (1976), 306--381.

\bibitem{RRZ}
M. Reyes, D. Rogalski and J.J. Zhang,
Skew Calabi-Yau algebras and homological identities,
\emph{Adv. Math.} \textbf{264} (2014), 308--354.



\bibitem{Smith}
S.P. Smith,
Some finite dimensional algebras related to elliptic curves,
\emph{Representation theory of algebras and related topics} (Mexico City, 1994), 315--348, 
CMS Conf. Proc. {\bf 19}, Amer. Math. Soc., Providence, RI, 1996.

\bibitem{Toen}B. To\"en, 
Derived algebraic geometry,
\emph{EMS Surv. Math. Sci.} {\bf 1} (2014), no. 2, 153--240.

\bibitem{TVII}
B. To\"en and G. Vezzosi,
\emph{Homotopical algebraic geometry. II. Geometric stacks and applications},
{Memoirs of the AMS}, \textbf{193} (2008), no. 902.

\bibitem{VdBExist}
M. Van den Bergh, 
Existence theorems for dualizing complexes over non-commutative graded and filtered rings,
\emph{J. Algebra} \textbf{195} (1997), 662--679.

\bibitem{VdB98}
M. Van den Bergh,
A relation between Hochschildt homology and cohomology for Georenstein rings,
\emph{Proc. Amer. Math. Soc.} \textbf{126} (1998), 1345--1348,
and Erratum, \emph{Proc. Amer. Math. Soc.} \textbf{130} (2002), 2809--2810.


\bibitem{VdBDP}
M. Van den Bergh, Double Poisson algebras.
\emph{Trans. Amer. Math. Soc.} \textbf{360} (2008),  no. 11, 5711--5769. 

\bibitem{VdBNCHam}
M. Van den Bergh, 
Non-commutative quasi-Hamiltonian spaces.
Poisson geometry in mathematics and physics,
\emph{Contemp. Math.} \textbf{450}, 273--299, Amer. Math. Soc., Providence, RI, 2008.

\bibitem{Yekutieli}A. Yekutieli,
\emph{Derived categories}. Cambridge Studies in Advanced Mathematics {\bf 183}. 
Cambridge University Press, Cambridge, 2020.

\bibitem{YZ}
A. Yekutieli and J.J. Zhang,
Homological transcendence degree,
\emph{Proc. London Math. Soc.} \textbf{93} (2006), 105--137.


\bibitem{Yeung}
W.-K. Yeung, 
Pre-Calabi-Yau structures and moduli of representations,
arXiv:1802.05398v3.

\bibitem{Yeung2}
W.-K. Yeung, 
Shifted symplectic and Poisson structures on global quotients,
arXiv:2103.09491v2.


\end{thebibliography}
\end{document}